\documentclass[11pt]{amsart}

\usepackage{amsmath}
\usepackage{datetime}
\usepackage{amsthm}
\usepackage{mathrsfs}
\usepackage{graphicx}
\usepackage[usenames, dvipsnames]{xcolor}
\definecolor{lime}{rgb}{0.2, 0.8, 0.2}
\def\bel{\begin{equation}\label}
\def\eeq{\end{equation}}
\def\bel{\begin{equation}\label}
\def\eeq{\end{equation}}
\newtheorem{Definition}{Definition}[section]
\newtheorem{Theorem}{Theorem}[section]
\newtheorem{Remark}{Remark}[section]
\newtheorem{Lemma}{Lemma}[section]

\newtheorem{Proposition}{Proposition}[section]
\newtheorem{Corollary}{Corollary}[section]
\newtheorem{Example}{Example}[section]

\def\bel{\begin{equation}\label}
\def\eeq{\end{equation}}
\begin{document}
\title
{Unbounded variation and solutions  
 of impulsive control systems}
 \thanks{ This research is partially supported by  the Gruppo Nazionale per l' Analisi Matematica, la Probabilit\`a e le loro Applicazioni (GNAMPA) of the Istituto Nazionale di Alta Matematica (INdAM), Italy  and 
by Padova University grant PRAT 2015 ``Control
of dynamics with reactive constraints''}

 \textwidth=135mm
\def\fudge{\mathchoice{}{}{\mkern.5mu}{\mkern.8mu}}
\def\bbc#1#2{{\rm \mkern#2mu\vbar\mkern-#2mu#1}}
\def\bbb#1{{\rm I\mkern-3.5mu #1}}
\def\bba#1#2{{\rm #1\mkern-#2mu\fudge #1}}
\def\bb#1{{\count4=`#1 \advance\count4by-64 \ifcase\count4\or\bba A{11.5}\or0
\bbb B\or\bbc C{5}\or\bbb D\or\bbb E\or\bbb F \or\bbc G{5}\or\bbb H\or
\bbb I\or\bbc J{3}\or\bbb K\or\bbb L \or\bbb M\or\bbb N\or\bbc O{5} \or
\bbb P\or\bbc Q{5}\or\bbb R\or\bbc S{4.2}\or\bba T{10.5}\or\bbc U{5}\or%
\bbb P\or\bbc Q{5}\or\bbb R\or\bba S{8}\or\bba T{10.5}\or\bbc U{5}\or
\bba V{12}\or\bba W{16.5}\or\bba X{11}\or\bba Y{11.7}\or\bba Z{7.5}\fi}}
\def \ol{\overline}
\def \Z {{\bb Z}}
\def \R {{\bb R}}
\def \SS {{\mathcal  T}}
\def \K {{\mathcal  K}}
\def \C {{\mathcal  C}}
\def \Sf {{\bf S}}
\def \N {{\bb N}}
\def \Q {{\bb Q}}
\def \vsm{\vskip 0.2 truecm}
\def \vv{\vskip 0.5 truecm}
\def \vs {{\vskip 1 truecm}}
 \def\L{\mathcal{L}}
  \def\U{\Gamma}
\def\bel{\begin{equation}\label}
\def\eeq{\end{equation}}
\def\ab{\overline{a}}
\def\cb{\overline{c}}
\def\db{\overline{d}}
\def\eb{\overline{e}}
\def\fb{\overline{f}}
\def\gb{\overline{g}}
\def\hb{\overline{h}}
\def\ib{\overline{i}}
\def\jb{\overline{j}}
\def\kb{\overline{k}}
\def\lb{\overline{l}}
\def\mb{\overline{m}}
\def\nb{\overline{n}}
\def\ob{\overline{o}}
\def\pb{\overline{p}}
\def\qb{\overline{q}}
\def\rb{\overline{r}}
\def\sb{\overline{s}}
\def\tb{\overline{t}}
\def\ub{{\overline{u}_0}}
\def\vb{\overline{v}}
\def\wb{\overline{w}}
\def\xb{{\overline{x}_0}}
\def\yb{\overline{y}}
\def\zb{\overline{z}}
\def\SS{\mathcal S}
\def\gnu{\bf GNU}
\def\ah{\hat{a}}
\def\bh{\hat{b}}
\def\ch{\hat{c}}
\def\dh{\hat{d}}
\def\eh{\hat{e}}
\def\fh{\hat{f}}
\def\gh{\hat{g}}
\def\hh{\hat{h}}
\def\ih{\hat{i}}
\def\jh{\hat{j}}
\def\kh{\hat{k}}
\def\lh{\hat{l}}
\def\mh{\hat{m}}
\def\nh{\hat{n}}
\def\oh{\hat{o}}
\def\ph{\hat{p}}
\def\qh{\hat{q}}
\def\rh{\hat{r}}
\def\sh{\hat{s}}
\def\th{\hat{t}}
\def\uh{\hat{u}}
\def\vh{\hat{v}}
\def\wh{\hat{w}}
\def\xh{\hat{x}}
\def\yh{\hat{y}}
\def\zh{\hat{z}}
\def\i{\item}
\def\ds{\displaystyle}
\def\n{\noindent}
\def\b{\bullet}
\def\bega{\begin{array}}
\def\enda{\end{array}}
\def\d{ {\rm d} }
\author{Monica Motta}
\address{M. Motta, Dipartimento di Matematica,
Universit\`a di Padova\\ Via Trieste, 63, Padova  35121, Italy\\
Telefax (39)(49) 827 1499,\,\, Telephone (39)(49) 827 1368
email\,
motta@math.unipd.it}

\author{Caterina Sartori}
\address{C. Sartori, Dipartimento di Matematica,
Universit\`a di Padova\\ Via Trieste, 63, Padova  35121, Italy\\Telefax (39)(49) 827 1499,\,\, Telephone (39)(49) 827 1318
email\,sartori@math.unipd.it}
\date{\today}

\begin{abstract} 
We consider a control system   with  dynamics   which are  affine in the  (unbounded) derivative of the control $u$.   We introduce a notion of  generalized solution  $x$  on $[0,T]$ for controls $u$ of   bounded  total variation  on $[0,t]$ for every $t<T$, but  of  possibly infinite   variation  on $[0,T]$.  This solution  has a simple representation formula based on  the so-called graph completion approach, originally developed for BV controls.
  We  prove the well-posedness of this  generalized solution  by showing that   $x$ is a limit solution, that is  the pointwise limit of regular trajectories of the system. In particular, we single out  the subset of limit solutions  which is in one-to-one correspondence with the set of  generalized solutions.  
The controls that we consider   provide  the natural setting   for treating    some questions on  the controllability of the system and some  optimal control problems with  endpoint constraints  and lack of coercivity.  
 \end{abstract}
\subjclass[2010]{49N25, 34H05, 93C10, 49K40, 49J15 } 
\keywords{  Impulsive control, generalized solutions, well posedness,  optimal control with unbounded variation}
\maketitle 
\section*{Introduction} 
 We consider a control system  of the form
\bel{E}
 \dot x (t)= g_0(x(t),u(t),v(t)) +  \sum_{i=1}^m  {g}_i(x(t),u(t))\dot u_i(t),\quad t\in]0,T], 
 \eeq
\bel{EO}
 x(0)=\xb,\quad u(0)=\ub,
\eeq
where $x\in\R^n$ and the measurable control pair $(u,v)$   ranges over a compact set $U\times V\subset\R^m\times\R^q$.   Due to the presence of the derivatives $\dot u_i$, (\ref{E}) is a so-called impulsive  control system,  where a solution $x$ can be  provided by the usual  Carath\'eodory solution only if $u$ is an absolutely continuous  control. For less regular $u$,  several concepts of solutions have been introduced in the literature, either for   commutative  systems, where the Lie brackets $[({\bf e}_i,g_i),({\bf e}_j,g_j)]=0$ for all $i,j=1,\dots,m$  (see e.g.  \cite{BR1}, \cite{D}, \cite{Sa}, \cite{AR}), or assuming  that  $u$ and $x$ are  functions of bounded variation, when the Lie Algebra is non trivial (see e.g. \cite{BR},  \cite{MR}). These  solutions  are described by different authors in fairly equivalent ways, and we will refer to them as graph completion solutions, since they are obtained  completing the graph of $u$ (see e.g., \cite{Ri}, \cite{Wa}, \cite{GS},\cite {KDPS}, \cite{SV},  \cite{WS},   \cite{AKP}, \cite{K},  \cite{PS},   \cite{MS}, \cite{BP}, \cite{MiRu}, for numerical approximations \cite{CF}, for extensions to stochastic control \cite{MS1}, \cite{DMi}).
 In the less studied noncommutative case  with  controls $u$ of unbounded variation, let us  mention the notion of  looping controls \cite{BR2},   the  definition of  limit solution  \cite{AR}, and the theory of rough paths  (for continuous  $u$)   \cite{LQ}.
Differently from the cases of commutative systems and of bounded variation controls $u$, in the general case   no   (simple) explicit representation formula  of the solution  is  known.  

\vsm
In this paper  we focus on the noncommutative case for  controls $u:[0,T]\to U$ with total variation bounded on $[0,t]$ for every $t<T$ but possibly infinite on $[0,T]$, in short  $u\in \ol{BV}_{loc}(T)$.
 We   extend the graph completion approach to such controls  and for any $u\in \ol{BV}_{loc}(T)$  and measurable $v$,  we introduce a notion of     solution $x$  to \eqref{E}--\eqref{EO} on $[0,T]$,  which we  call  {\it {\rm BV}$_{loc}$ graph completion  solution}  (see Definitions \ref{GCloc}, \ref{GClocS}).
In particular,  we first  define an {\it {\rm AC}$_{loc}$ solution}  $x$ on $[0,T]$, obtained by extending  $(x,u)$ to be absolutely continuous on $[0,t]$ for $t<T$   to $[0,T]$,  by choosing $(x,u)(T)=\lim_j(x,u)(\tau_j)$ for some sequence $\tau_j\nearrow T$.   Hence we prove that the concept of   {\it {\rm BV$_{loc}$} graph completion  solution}  $x$ is:
\vsm
i)   {\it well defined}, since for any $u\in\ol{BV}_{loc}(T)$ and measurable $v$ a corresponding  a  BV$_{loc}$ graph completion solution does exist (Theorem \ref{Egc});

ii)  {\it consistent} with that of  AC$_{loc}$ solution, in the sense that  if the pair $(x,u)$  is absolutely continuous on $[0,t]$ for $t<T$ 
and $x$   is a  BV$_{loc}$ graph completion solution, then $x$ is an AC$_{loc}$ solution (Theorem \ref{TST1});

iii) {\it well posed}, since  $x$ is the pointwise limit of  Carath\'eodory solutions $x_k$ to \eqref{E}, \eqref{EO} corresponding to inputs $(u_k,v)$, with  the controls $u_k$   absolutely continuous on $[0,T]$ and  pointwisely converging  to $u$. In this sense it is a  simple limit solution, as  recently defined in \cite{AR}  (see Definition \ref{edsdef}).  Actually, in Theorem \ref{TST2} we prove something more, in that  we characterize the specific subclass of  simple limit solutions, that we call  {\it {\rm BV$_{loc}S$} limit solutions},  corresponding to   BV$_{loc}$ graph completion  solutions. 
 
\vsm
 With respect to more general concepts,  the BV$_{loc}$ graph completion  solution  has  a nice representation formula, suitable to derive  necessary and sufficient optimality conditions for several optimization problems,   both in terms of  Pontrjagin Maximum Principle and of  Hamilton-Jacobi-Bellman equations (some results  in the last direction have been already obtained in \cite{MS2}).     Moreover, 
controls $u\in \ol{BV}_{loc}(T)$  are  relevant  in controllability issues, like approaching a target set,  and  in  optimization problems with endpoint constraints and certain  running costs lacking coercivity, as  in the following example  (see also Example \ref{EIntr}).
 
\begin{Example}  {\rm Let $\C\subset\R^n\times U$ be a closed subset, the {\it target},  and  let ${\bf d}(\cdot)$ denote  the Euclidean distance from $\C$.  Let us minimize
 \bel{value} 
  \int_0^{T} [\ell_0(x(t),u(t),v(t))+\ell_1(x(t),u(t))\,|\dot u|]\, dt, 
 \eeq
over  trajectory-control pairs $(x,u,v)$ of \eqref{E}, \eqref{EO} such that
\bel{IntrEPC}
{\bf d}((x(t),u(t)))>0 \ \forall t<T, \quad \liminf_{t\to {T}^-}{\bf d}((x(t),u(t)))=0,
\eeq
assuming that  $\ell_0\ge0$ and $\ell_1$ verifies  
$$
 \ell_1(x,u)\ge c({\bf d}(x,u)),  
$$
for some strictly increasing, continuous function  $c:\R_+\to\R_+$ with $c(0)\ge0$.  In this case,  only controls $u\in \ol{BV}_{loc}(T)$ may have finite cost. The above hypothesis on  $\ell_1$ generalizes the so-called {\it weak coercivity condition} $\ell_1\ge C_1>0$, assumed in many applications in order to rule out controls with unbounded variation.  Notice that,  as the variation of  $u$ is unbounded,  we expect chattering phenomena  as $t$ tends to $T$ (see e.g.  \cite{CGPT} and the references therein), which in impulsive control systems will affect  both   $u$ and   $x$. It is thus natural to replace the usual endpoint condition  $(x(T),u(T))\in\C$   with \eqref{IntrEPC} (see   Remark \ref{Remmin}).   }
 \end{Example} 
 
The paper is organized as follows. 
We end this section with some notation and the precise  assumptions that are needed in the paper.  In Section \ref{prel} we define  AC$_{loc}$ solutions and introduce  the notion  of BV$_{loc}$ graph completion solution.   Existence of  such   solutions   and their consistency  with regular,  AC$_{loc}$ solutions are established  in Section \ref{SsTh}.  In Section \ref{SLS} we define BV$_{loc}S$ limit solutions  and in Section \ref{ASS}  we obtain our main result:  the equivalence between  BV$_{loc}$ graph completion solutions and  BV$_{loc}S$ limit solutions.  Section \ref{proofs} is devoted to the proofs of some technical results.
 
\subsection{Notation}  Let  $E\subset \R^N$. For any $f:[a,b]\to E$,  $Var_{[a,b]}(f)$   denotes the (total) variation of $f$ on $[a,b]$. When $E$ is bounded, we call {\it diameter} of $E$ the value diam$(E):=\sup\{|u_1-u_2|: \ u_1, \, u_2 \in E\}$. 
  For $T>0$, let $AC([0,T],E)$, $BV([0,T],E)$ denote the set of absolutely continuous and BV   functions \newline $f:[0,T]\to E$, respectively,  and let   us set 
$$\begin{array}{l}
AC_{loc}([0,T[,E):=
\{f\in AC([0,t],E)   \,\forall \,t<T, \ \lim_{t\to T}Var_{[0,t]}[f]\le+\infty\}, 
\end{array}
$$
$$\begin{array}{l}
 BV_{loc}([0,T[,E):=
\{f\in BV([0,t],E)   \, \forall t<T , \   \lim_{t\to T}Var_{[0,t]}[f]\le+\infty \}.
 \end{array}
$$
The set $L^1([0,T], E)$ is the usual quotient with respect to the Lebesgue measure. 

\noindent  When no confusion on the codomain may arise, in what follows   in place of the above sets we will simply write $AC(T)$, $BV(T)$, ${AC}_{loc}(T)$, ${BV}_{loc}(T)$, and  $L^1(T)$, respectively.

We set  $\R_+:=[0,+\infty[$ and  call   {\it modulus } (of continuity) any increasing, continuous function $\omega:\R_+\to\R_+$ such that $\omega(0)=0$ and  $\omega(r)>0$ for every $r>0$.

\subsection{Assumptions} Throughout the paper we assume the following hypotheses:
\begin{itemize}
 \item[(i)]
the sets $U \subset \R^m$  and $V \subset \R^l$ are compact;
 \item[(ii)] the control vector field $g_0: \R^n \times U \times V\to\R^n$  is continuous and 
 $ (x,u) \mapsto  g_0(x,u,v)$
 is locally Lipschitz on $\R^n \times U$,  uniformly in $v \in V$;
  \item[(iii)] for each $i=1,\dots,m$  the control vector field  $g_i: \R^n \times U\to\R^n$  is locally Lipschitz continuous;
  \item[(iv)] there exists $M>0$ such that
  $$
  \left| \Big(g_0(x,u,v),g_1(x,u),\dots,g_m(x,u)\Big)\right| \leq  M(1+|(x,u)|),$$
  for every $(x,u,v) \in \R^n\times U\times V$.
  \end{itemize}
 
In the main results  we will use  the following condition. 
\begin{Definition}[Whitney property]\label{W}  
A  compact subset  $U\subset \R^m$ has the Whitney property if there is some $C\ge1$ such that for every pair $u_1$, $u_2 \in U$, there exists an absolutely continuous path $\tilde u: [0, 1] \to U$ verifying
\bel{hCA}
\tilde u(0)=u_1, \quad \tilde u(1)=u_2, \quad Var[\tilde u]\le C|u_1-u_2|.
\eeq
\end{Definition}

For instance,  compact, star-shaped sets verify the Whitney property.

\vv
\section{BV$_{loc}$ graph completion solutions}\label{prel}
 For any control $(u,v)\in {AC}_{loc}(T)\times  {L^1}(T)$ with $u(0)=\ub$,   let
$$x=x[\xb,\ub,u,v]$$
denote the unique   Carath\'eodory solution to   (\ref{E})--(\ref{EO}), defined on $[0,T[$. 
 
\vsm
\subsection{AC$_{loc}$ controls and solutions}\label{SsACloc}

Let us introduce the set of controls $u\in{AC}_{loc}(T)$ {\it extended to $[0,T]$}:
\bel{barloc}
\ol{AC}_{loc}(T):=\left\{u\in {AC}_{loc}(T): \  u(T):=\lim_ju(\tau_j), \ \text{for some} \  \tau_j\nearrow T\right\} 
\eeq
and the corresponding extended solutions: 
  
 \begin{Definition}[AC$_{loc}$  solution]\label{DACloc} Let $(u,v)\in \ol{AC}_{loc}(T)\times L^1(T)$ with $u(0)=\ub$, and set  $x:=x[\xb,\ub,u,v]$.  When $x$ is bounded on $[0,T[$, we introduce an {\rm extension of $x$ to $[0,T]$,}   such that   
\bel{xuset}
 (x(T),u(T))\in(x,u)_{set}(T):=\{\lim_j(x,u)(\tau_j), \ \text{for some} \  \tau_j\nearrow T\}.  
\eeq 
We call  $x$ a {\rm (single-valued)  AC$_{loc}$ solution  on $[0,T]$} and $(x,u,v)$ an {\rm AC$_{loc}$ trajectory-control pair}.   
\end{Definition}
Clearly, the extension of $(x,u)$ to $[0,T]$ is not unique, in general.
\begin{Remark}\label{Remmin}{\rm In order to motivate the above extension,  let
us consider AC$_{loc}$  trajectory-control pairs $(x,u,v)$ defined on $[0,T]$ as above,   verifying the final constraint 
\bel{TC}
(x,u)(T)\in \C,
\eeq  
 where  $\C\subset\R^n\times U$ is a closed set, which we call the {\it target}. Condition \eqref{TC} turns out to be verified when  $(x,u)_{set}(T)\cap\C\ne\emptyset$ and this is  equivalent to have
$$
\liminf_{t\to T^-} d\left((x(t),u(t)),\C\right)=0.
$$
Incidentally,   the stronger condition $(x,u)_{set}(T)\subseteq\C$ is instead equivalent to  
\bel{sTC}
\lim_{t\to T^-} d\left((x(t),u(t)),\C\right)=0 
\eeq
and this limit  holds true if and only if for {\it every} increasing sequence $(\tau_j)_j$ converging to $T$ there exists a subsequence such that  $\lim_{j'}(x(\tau_{j'}),u(\tau_{j'}))=(\bar x,\bar u)\in \partial\C$.   Definition \ref{DACloc} can be easily adapted to applications where \eqref{TC} has to be interpreted as in \eqref{sTC}. }

\end{Remark}

 \vsm
\subsection{Space-time controls and solutions}\label{SsST}
 
For $L>0$ and $0<S\le+\infty$,    let ${\mathcal U}_L(S)$  denote the subset of $L$-Lipschitz  maps 
 $$(\varphi_0,\varphi):[0,S[\to \R_+\times U,$$ 
 such  that $\varphi_0(0)=0$, and $\varphi_0'(s) \geq 0$, $\varphi_0'(s)+|\varphi' (s)|\le L$ 
for   almost every  $s\in[0,S[$;   the apex  $'$  denotes differentiation with respect to the  {\it pseudo-time} $s$. Let ${\mathcal M}(S)$  denote the set of measurable functions  $\psi:[0,S[\to V$.

\begin{Definition}  [Space-time control and solution]   \label{STdef}
We   will  call  {\em space-time controls} the elements $(\varphi_0,\varphi,\psi, S)$, where $0<S\le+\infty$ and  $(\varphi_0,\varphi,\psi)$ belongs to the   set $\bigcup_{L>0}{\mathcal U}_L(S)\times {\mathcal M}(S)$.
\newline  Given  $(\xb,\ub)\in \R^n\times U$ and a space-time control $(\varphi_0,\varphi,\psi,S)$ such that $\varphi(0)=\ub$, the {\rm space-time control system}  is defined by
\bel{SEaff}
\left\{
\begin{array}{l}
\xi'(s)  =  g_0(\xi(s),\varphi(s),\psi(s)) \varphi_0'(s)+ \sum_{i=1}^m {  g}_i(\xi(s),\varphi(s)) {\varphi'_i}(s)  \quad s\in]0,S[, \\\,\\
 \xi(0)=\xb.
\end{array}
\right.
\eeq
 We will  write  $\xi [\xb,\ub,\varphi_0,\varphi,\psi]$ to denote  the solution of  \eqref{SEaff}. 
 \end{Definition}
 
Space-time controls and solutions can be seen as an {\it extension} of regular, that is  AC  and  AC$_{loc}$,   controls and solutions. Indeed, if  instead of a   control pair $(u,v)\in \ol{AC}_{loc}(T)\times L^1(T)$ we consider any  time-reparametrization $t=\varphi_0(s)$ of its graph $(t,u(t),v(t))$, we obtain a space-time control $(\varphi_0,\varphi,\psi):=(\varphi_0, u\circ\varphi_0,v\circ\varphi_0)$ \footnote{Since  every $L^1$ equivalence class contains  Borel measurable representatives, here and in the sequel we  tacitly assume that the maps $v$ and $\psi$ are Borel measurable when necessary.} and the corresponding space-time solution  $\xi [\xb,\ub,\varphi_0,\varphi,\psi]$ is nothing but  $x[\xb,\ub,u,v]\circ\varphi_0$.  On the other hand, space-time controls $(\varphi_0,\varphi,\psi)$ such that $(\varphi,\psi)$  evolves  on the intervals where $\varphi_0$ is constant, are  more general objects than the graphs of a control   $(u,v)$  with $u$ in $AC(T)$ or in $\ol{AC}_{loc}(T)$  (see Proposition \ref{REmb} and Theorem \ref{ThAC}).

\noindent In addition,  the space-time system has a {\it  parameter-free character}.  Precisely, if $(\varphi_0,\varphi,\psi,S)$, $(\tilde\varphi_0,\tilde\varphi,\tilde\psi,\tilde S)$ verify  $(\varphi_0,\varphi,\psi)=(\tilde\varphi_0,\tilde\varphi,\tilde\psi)\circ \tilde s$  for some reparametrization $\tilde s:[0,  S]\to[0,\tilde S]$,  it can be  shown that $\xi=\tilde\xi\circ\tilde s$, if $\xi$ and $\tilde\xi$ denote the solutions to (\ref{SEaff}) corresponding to $(\varphi_0,\varphi,\psi)$ and $(\tilde\varphi_0,\tilde\varphi,\tilde\psi)$, respectively.   For these reasons, we  consider the following subset of space-time controls.    

 \begin{Definition}[Feasible space-time controls]\label{fstc}   We   call  {\rm feasible} the space-time controls belonging to the subset 
\bel{fc}
\begin{array}{l}
\U(T;\ub):= 
\left\{(\varphi_0,\varphi,\psi, S): \  0<S\le+\infty, \ (\varphi_0,\varphi,\psi)\in {\mathcal U}_1(S)\times {\mathcal M}(S),\qquad
\qquad\qquad
\qquad\qquad
\qquad\right.\\  \,\\\left. \
\qquad
\qquad\, 
\qquad \varphi_0'(s)+|\varphi' (s)|=1 \text{ a.e.},   \ \  \varphi(0)=\ub, \ \ \lim_{s\to S}\varphi_0(s)=T\right\}.
\end{array}
\eeq
\end{Definition}
 
For any feasible  space-time control $(\varphi_0,\varphi,\psi, S)$,  the pseudo-time  $s$  coincides with  the {\it arc-length parameter } of the curve $(\varphi_0,\varphi)$ (with respect to the norm $\varphi_0'(s)+|\varphi' (s)|$)  and we have the identity 
\bel{ids}
s=\varphi_0(s)+Var_{[0,s]}[\varphi] \quad \forall s\in[0,S[.
\eeq
As a consequence,  the final pseudo-time is $S=T+Var_{[0,S[}[\varphi]$ and  
$$
S=+\infty \ \text{ if and only if } \  Var_{[0, S[}[\varphi]=+\infty. 
$$

Let us introduce the following notion of {\it feasible space-time trajectory-control pair}     extended to  the {\it closed} set $[0,S]$, even in case $S=+\infty$.   
  \begin{Definition} [Feasible space-time trajectory-control pairs]\label{fsttc}  Let  $(\varphi_0,\varphi,\psi, S)\in \U(T;\ub)$ be a feasible space-time control and set $\xi:=\xi [\xb,\ub,\varphi_0,\varphi,\psi]$. If $S<+\infty$, we extend  $(\xi,\varphi_0,\varphi)$ to $[0,S]$ by continuity.  If $S=+\infty$  and  $\xi$   is bounded, we introduce an {\rm  extension of $(\xi,\varphi)$ to $[0,+\infty]$,} such that
  \bel{infty}
    (\xi,\varphi)(+\infty)\in(\xi,\varphi)_{set}(+\infty):=\{\lim_j(\xi,\varphi)(s_j):  
  \  \text{for some  \ $s_j\nearrow+\infty$ } \},
  \eeq
and call $(\xi,\varphi_0,\varphi,\psi, S)$ a {\rm (single-valued) feasible space-time trajectory-control pair on $[0,S]$.} 
 \end{Definition}
 
The next results are easy consequences of the chain rule.
\begin{Proposition}\label{REmb}   
(i) \,  Given $(u,v)\in \ol{AC}_{loc}(T)\times L^1(T)$ with $u(0)=\ub$,   set $x:=x[\xb,\ub,u,v]$ and
\bel{Rep}
\begin{array}{l}
 \sigma(t) {:=}  \int_0^t (1 +|\dot u(\tau)|) d\tau \quad\forall t\in[0,T[, \  \ S:=\lim_{t\to T}\sigma(t) \ (\le+\infty) \\ \, \\
\varphi_0{:=}  \sigma^{-1},\ \ \varphi{:=} u\circ \varphi_0,  \ \ \psi{:=}  v \circ \varphi_0 ,  \ \xi:=\xi[\xb,\ub,\varphi_0,\varphi,\psi]   \ \ \text{in }  \ [0,S[.
\end{array}
\eeq
Then   $(\xi,\varphi_0,\varphi,\psi, S)$ is a feasible space-time trajectory,  $(\xi,\varphi,\psi)\circ\sigma=(x,u,v)$ and, when $u\in \ol{AC}_{loc}(T)\setminus AC(T)$ (so that $S=+\infty$) and $x$ is bounded,   $(\xi,\varphi)_{set}(+\infty)=(x,u)_{set}(T)$. In particular,  if  $(x,u)(T) =\lim_j(x,u)(\tau_j)$ for some $\tau_j\nearrow T$,     we have $\lim_j \sigma(\tau_j)=+\infty$ and  we can set
 $$
(\xi,\varphi)(+\infty):=\lim_j(\xi,\varphi)(\sigma(\tau_j))=(x,u)(T).
 $$
  
\noindent (ii) Vice-versa, given $(\varphi_0,\varphi,\psi, S)\in \U(T;\ub)$   with 
$$
\varphi_0'(s)>0 \ \text{for a.e. } \ s\in[0,S[,
$$
let us set $\xi:=\xi[\xb,\ub,\varphi_0,\varphi,\psi]$ and
$$
 \quad (u,v):=( \varphi,\psi)\circ\varphi_0^{-1}, \quad x:=x[\xb,\ub,u,v].
 $$

\noindent  Then $(x,u,v)$ is a  trajectory-control pair of \eqref{E}--\eqref{EO},    $(x,u,v)\circ\varphi_0=(\xi,\varphi,\psi)$,  and, when  $S=+\infty$ and $\xi$ is bounded,   $(x,u)_{set}(T)=(\xi,\varphi)_{set}(+\infty)$. In particular,    if $(\xi,\varphi)(+\infty)=\lim_j(\xi,\varphi)(s_j)$ along some   $s_j\nearrow+\infty$, we have    $\lim_j\varphi_0(s_j)=T$ and we can set
$$
(x,u)(T):=\lim_j(x,u)(\varphi_0(s_j))=(\xi,\varphi)(+\infty).
$$
  \end{Proposition}

Owing to Proposition \ref{REmb} we can identify any AC$_{loc}$  trajectory-control pair with the associated feasible space-time trajectory-control pair:
\begin{Definition}[Arc-length parametrization]\label{Dal}   We call {\rm arc-length   graph- \linebreak parametrization} of an AC$_{loc}$  trajectory-control pair $(x,u,v)$ the feasible space-time trajectory-control pair  $(\xi,\varphi_0,\varphi,\psi, S)$ defined by \eqref{Rep}. 
\end{Definition}

Proposition \ref{REmb}  also  implies  the following equivalence result.
\begin{Theorem}\label{ThAC}
The set  of {\rm AC}  [resp., {\rm AC}$_{loc}\setminus$ {\rm AC}] trajectory-control pairs of   \eqref{E}-\eqref{EO}   is in one-to-one correspondence with the subset of feasible space-time trajectory-control pairs $(\xi,\varphi_0,\varphi,\psi, S)$ with  $S<+\infty$  [resp.,  $S=+\infty$] and $\varphi'_0>0$ a.e..   
\end{Theorem}

\subsection{BV$_{loc}$ graph completions}\label{SsGC}
 
 Let us introduce the basic notions of the graph completion approach, which  originally was dealing with inputs $u\in BV(T)$    and that  we now extend to  controls  $u\in \ol{BV}_{loc}(T)$, where
$$
\ol{BV}_{loc}(T):= \{u: \  u:[0,T]\to U, \ \ u\in {BV}_{loc}(T)\}.
$$
 We refer to   \cite{BR} for the definition and some basic results on BV graph completions, to \cite{MR} for  BV graph completions with dependence on the ordinary control $v$ and  to \cite{AR}, \cite{AMR} for  the concept of clock.
\begin{Definition}[Graph completion and clock]\label{GCloc}  Let  $(u,v) \in \ol{BV}_{loc}(T) \times L^1(T)$  and $u(0)=\ub\in U$.  We say that a   space-time control     $(\varphi_0,\varphi,\psi, S)\in\U(T; \ub)$ with $S\le+\infty$,    is a {\em BV$_{loc}$ graph completion   of $(u,v)$}  if 
\begin{itemize} 
\item[i)]  $\forall t\in[0,T[$,  $\exists s\in[0,S[$ such that
$(\varphi_0,\varphi,\psi)(s)=(t,u(t),v(t))$; 
\item[ii)] when $S<+\infty$, $(\varphi_0,\varphi)(S)=(T,u(T))$;
\item[iii)] when  $S=+\infty$,   
\bel{gcu}
 \lim_j\varphi(s_j)=u(T) \quad\text{for some }s_j\nearrow +\infty.
\eeq 
In this case we will write, in short,  $(\varphi_0,\varphi)(+\infty)=(T,u(T))$.
\end{itemize}
   We call a {\rm clock} any increasing function $\sigma:[0,T]\to[0,S]$ such that 
$$
(\varphi_0,\varphi)(\sigma(t))=(t,u(t)) \ \text{ for every $t\in[0,T]$, $\sigma(0)=0$ and $\sigma(T)=S$.} \footnote{When $S=+\infty$, the notation  $\sigma(T)=+\infty$ means just that   $(\varphi_0,\varphi)(+\infty)=(T,u(T))$  in the sense of (\ref{gcu}), but it might be $\lim_{t\to T^-}\sigma(t)<+\infty$.} 
$$
\end{Definition}

 If $(\varphi_0,\varphi,\psi,S)$  is a BV$_{loc}$ graph completion of  a control $(u,v) \in \ol{BV}_{loc}(T) \times L^1(T)$,  then $Var_{[0,T]}[u]\le Var_{[0,S[}[\varphi]$.  Indeed, $(\varphi_0,\varphi)$ is    a parametrization of  a {\it completion} of $(t,u(t))$, where, roughly speaking,  a discontinuity of $u$ at $\bar t$ is  bridged by an arbitrary continuous curve  in $\{\bar t\}\times U$. Therefore, if $S<+\infty$ the control $u$ has necessarily bounded  variation $Var_{[0,T]}[u]\le Var_{[0,S[}[\varphi]$, while when  $S=+\infty$, $Var_{\R_+}[\varphi]=+\infty$ but   the  control  $u$ may belong either to $\ol{BV}_{loc}(T)$ or to BV$(T)$.

 \begin{Definition}[Graph completion solution]\label{GClocS}   Let   $(\varphi_0,\varphi,\psi, S)$ be a  BV$_{loc}$ graph  completion    of $(u,v) \in \ol{BV}_{loc}(T) \times L^1(T)$  with $u(0)=\ub$, let $\sigma$ be a clock    and set $\xi:=\xi[\xb, \ub,\varphi_0,\varphi,\psi]$.  When $S=+\infty$, let us suppose that $\xi$ is bounded.
 
We define a {\em BV$_{loc}$ graph completion  solution }  to \eqref{E}-\eqref{EO}  associated to  $(\varphi_0,\varphi,\psi, S)$ and $\sigma$, a  map  
$$
x:[0,T]\to\R^n, \quad x(t){:=} \xi\circ\sigma(t)\  \ \forall t\in[0,T[,
$$
and
\begin{itemize} 
\item[i)] if $S<+\infty$,  $x(T)=\xi(S)$;
\item[ii)] if  $S=+\infty$,  $(x(T),u(T))\in  (\xi,\varphi)_{set}(+\infty)$ {\rm (see  \eqref{infty}).}
\end{itemize}
\end{Definition}

 Notice that graph completions allow for jumps of the trajectory even at the times $t$ where $u$ is continuous (a loop of u could be considered at these instants, which, owing to the non-triviality of the Lie algebra generated by \linebreak $\{({\bf e}_1,g_1), . . . , ({\bf e}_m,g_m)\}$, might determine a discontinuity in $x$). 
 
 \vsm
 Just for regular  controls $u\in AC_{loc}(T)$ let us consider the following,  more restrictive notion of graph completion,  where essentially 
 the variation added to $u$ by introducing loops is finite; in other words,  the difference between the  variation  of  the  completion of $(t,u(t))$ and that of $u$ is finite. This notion will play an important role in   Theorem \ref{Egc}.
   \begin{Definition}[Graph completion with BV loops]\label{GCc} 
Given a  BV$_{loc}$ graph  completion   $(\varphi_0,\varphi,\psi, S)$  of a control $(u,v) \in \ol{AC}_{loc}(T) \times L^1(T)$, we say that it is a  {\em    graph  completion with {\rm BV} loops} if either $S<+\infty$ or $S=+\infty$ and
\bel{len}
Var_{\R^+}(\varphi_{|\varphi'_0=0})=\text{meas}  \{s\in\R^+: \ \varphi'_0(s)=0\}<+\infty,
\eeq
  \end{Definition}

 For instance,  the arc-length  graph parametrization $(\varphi_0,\varphi,\psi,S)$, $S\le+\infty$   of $(u,v)$  with $u\in AC_{loc}(T)$,  is   a  graph completion with BV loops (actually, with no loops), since $\varphi_0'> 0$ a.e.. On the other hand,  every graph completion $(\varphi_0,\varphi,\psi,S)$ with $S=+\infty$ of a control $(u,v) \in AC(T) \times L^1(T)$  has not $BV$ loops.  

\section{Existence and consistency}\label{SsTh}
This section is devoted to prove the {\it existence} of BV$_{loc}$ graph completion solutions (Theorem \ref{Egc}),  and the {\it consistency} of such notion of solution with the  extended AC$_{loc}$  solutions considered in Subsection \ref{SsACloc} (Theorem \ref{TST1}).

  \begin{Theorem}[Existence]\label{Egc}   Let  $U$ have the Whitney property.  Then for  any  $(u,v)  \in \ol{BV}_{loc}(T)\times L^1(T)$, there exists  a BV$_{loc}$ graph completion $(\varphi_0,\varphi,\psi,+\infty)$, and,  for any clock $\sigma$,  there is an associated  BV$_{loc}$ graph completion  solution $x$ to \eqref{E}--\eqref{EO} on $[0,T]$.  
  \end{Theorem}
                                        
The following result, whose proof is postponed in Section \ref{proofs},  is the key point for the existence of BV$_{loc}$ graph completions with unbounded variation. 
 \begin{Lemma}\label{Legc} Let us assume that  $U$ has the Whitney property.  Then for any $u\in BV([a,b],U)$ and $\bar u_1\in U$,   there exist $\tilde S>0$ 
 and a $1$-Lipschitz continuous map $(\varphi_0,\varphi):[0,\tilde S]\to [a,b]\times U$ such that: 
\begin{itemize}
\item[(i)] $\varphi_0$   is increasing, $\varphi'_0+|\varphi'|=1$  a.e.,   $(\varphi_0,\varphi)(0)=(a,u(a))$,  
  for any $t\in[a,b]$  there is $s\in[0,\tilde S]$  such that $(\varphi_0,\varphi)(s)=(t,u(t))$, 
  $$
(\varphi_0,\varphi)(\tilde S)=(b,u(b)),  \quad{ and}\quad\exists \,\,  S\le \tilde S \, \,{\rm s.t.}\,\,\   (\varphi_0,\varphi)(S)=(b,\bar u_1). 
$$ 
Moreover, setting $V:= Var_{[a,b]}(u)$, one has
$$
   (b-a)+ V +|u(b)-\bar u_1|\le S\le \tilde S \le (b-a)+2C(V +|u(b)-\bar u_1|),
$$
  where $C$ is as in {\rm (\ref{hCA});}
\item[(ii)]  $(\varphi_0,\varphi)$ admits a $1$-Lipschitz continuous extension to $\R_+$ with
$\varphi'_0+|\varphi'|=1$  a.e. and $\lim_j(\varphi_0,\varphi)(s_j)=(b,\bar u_1)$,
along some increasing,  diverging sequence $(s_j)_j$.
\end{itemize}
  \end{Lemma}  
  
  \vsm 
\begin{proof}[Proof of Theorem \ref{Egc}]
   Let $(\bar t_i)_i\subset[0,T[$ be a strictly increasing sequence converging to $T$, with $\bar t_0=0$.  For every $i\ge1$, let us set $I_i:=[\bar t_{i-1},\bar t_i]$,  $|I_i|:=\bar t_i- \bar t_{i-1}$ and $V_i:=Var_{I_i}(u)$. Applying Lemma \ref{Legc} to the restriction  $u_{|_{I_i}}$  with $\bar u_1=u(T)$,  we can define   
$$
 0<S_i\le \tilde S_i, \qquad   (\varphi_{0_i},\varphi_i):[0,\tilde S_i]\to I_i\times U, 
$$
  such that $\varphi_{0_i}$   is increasing, $\varphi'_{0_i}+|\varphi_i'|=1$  a.e.,   
   $$
(\varphi_{0_i},\varphi_i)(0)=(\bar t_{i-1},u(\bar t_{i-1})),\  (\varphi_{0_i},\varphi_i) (\tilde S_i)= (\bar t_{i},u(\bar t_{i})), \ 
(\varphi_{0_i},\varphi_i)(S_i)=(\bar t_{i},u(T)), 
$$
and 
\bel{ric}
|I_i|+V_i+|u(\bar t_i)-u(T)|\le  S_i\le\tilde S_i \le|I_i|+2C(V_i+|u(\bar t_i)-u(T)|).
\eeq
Let us set, for $i\ge0$,   
$$
 \tilde S_0:=0, \quad  \tilde s_i:=\sum_{j=0}^{i}\tilde S_{j}, \quad  s_{i+1}:= \tilde s_{i}+S_{i+1},\quad   \tilde s_\infty=\lim_i\tilde s_i, 
$$
 and  let us introduce the  space-time control
$$
(\varphi_0, \varphi,\psi)(s):=\sum_{i=1}^{+\infty}(\varphi_{0_i},\varphi_i, v\circ\varphi_{0_i})(s-\tilde s_{i-1})\chi_{ [\tilde s_{i-1},\tilde s_i[}(s) \quad \forall s\in[0,\tilde s_\infty[,
$$
which  can be easily proved to be a BV$_{loc}$ graph completion of $(u,v)$.  If $\tilde s_\infty=+\infty$,  the proof  of the theorem is concluded. 
 In this case indeed,  $(\varphi_0,\varphi,\psi)$ is defined on $\R_+$, $\lim_i s_i=+\infty$    and 
\bel{lim}
\lim_{i}(\varphi_0,\varphi)(s_i)=(T,u(T)).
\eeq
 Incidentally, by 
  (\ref{ric}) this is always verified  when  $Var_{[0,T[}(u)=+\infty$. If instead $Var_{[0,T[}(u)<+\infty$ and $\tilde s_\infty<+\infty$,  we can extend  $(\varphi_0,\varphi,\psi)$ to a BV$_{loc}$ graph completion defined on  $\R_+$ and satisfying \eqref{lim} by Lemma \ref{Legc}, (ii).  
\end{proof}

\vsm

\begin{Theorem}[Consistency]\label{TST1}   Let  $(u,v)\in \ol{AC}_{loc}(T)\times L^1(T)$ with $u(0)=\ub$. Let   $x$  be   a {\rm BV}$_{loc}$  graph completion solution to  \eqref{E}-\eqref{EO} on $[0,T]$ belonging to $AC_{loc}([0,T[,\R^n)$.  Then
\begin{itemize}
\item[(i)] $x$ coincides with the  Carath\'eodory solution $x[\xb,\ub,u,v]$ on $[0,T[$;  

\item[(ii)]   $x$ is an {\rm AC}$_{loc}$ solution  to  \eqref{E}-\eqref{EO} on $[0,T]$  if  either $(x,u)$ is   associated to a   graph completion with {\rm BV} loops or only if $(x,u)(T)\in (x,u)_{set}(T)$  (see Definition \ref{DACloc}). 
\end{itemize}
\end{Theorem}   
 
 \vsm
Preliminarily, let us state the following uniform convergence result  for  space-time trajectory-control pairs on compact sets,  proven in Section \ref{proofs}. 
\begin{Proposition}\label{xiv} Let $T>0$,  $(\xb,\ub)\in\R^n\times U$  and $(\varphi_0,\varphi,\psi,S)\in {\mathcal {U}}(T; \ub)$ for some $S\le+\infty$. 
Assume that there exist $\tilde T\le T$, $\tilde S\le S$ with  $\tilde S<+\infty$    and a sequence  $(\varphi_{0_h},\varphi_h,\psi, S)_h\subset {\mathcal {U}}(T; \ub)$ such that, for every $h$,  
  $\varphi_{0_h}$ is strictly increasing,  $\varphi_{0_h}(\tilde S)=\tilde T$, and  
$$
\sup_{s\in[0,\tilde S]}|(\varphi_{0_h},\varphi_h)-(\varphi_{0},\varphi)|\to 0 \,\,{\text as}\, h\to+\infty.
$$
Let   $\sigma:[0,\tilde T]\to[0,\tilde S]$ be any increasing function such that $\varphi_0\circ\sigma(t)=t$ for every $t\in[0,\tilde T]$, $\sigma(0)=0$ and $\sigma(\tilde T)=\tilde S$. Then,  setting $v:=\psi\circ\sigma$,  we have 
\bel{xiv1}
\xi:=\xi[\xb,\ub,\varphi_0,\varphi,\psi]\equiv\xi[\xb,\ub,\varphi_0,\varphi,v\circ\varphi_0]
\eeq
 and, if  $\xi_h:=\xi[\xb,\ub,\varphi_{0_h},\varphi_h,v\circ\varphi_{0_h}]$, 
there exists a subsequence  (still denoted by $(\xi_h)_h$) such that
\bel{xivh}
\sup_{s\in[0,\tilde S]}|\xi_h(s)-\xi(s)|\to 0 \quad\text{as $h\to+\infty$.}
\eeq
\end{Proposition}

\vsm
\begin{proof}[Proof of Theorem    \ref{TST1}]
Given   $(u,v)\in \ol{AC}_{loc}(T)\times L^1(T)$ with $u(0)=\ub$ and  an associated  BV$_{loc}$  graph completion solution $x$ with $x\in AC_{loc}(T)$, let $(\varphi_0,\varphi,\psi, S)$ and $\sigma$, be a  BV$_{loc}$  graph completion  of $(u,v)$ and a clock, respectively, such that $x=\xi\circ\sigma$, where $\xi$ is the   space-time trajectory of \eqref{SEaff} corresponding to  $(\varphi_0,\varphi,\psi, S)$,  extended to $[0,S]$ as in Definition \ref{fsttc} (see also Definitions \ref{GCloc} and \ref{GClocS}).

The proof of part  (i) is a  generalization, with some simplifications,  of the proof of \cite[Theorem 2.2]{AR}, dealing  with BV inputs and trajectories.   For every $s\in[0,S[$,   let us consider the space-time control $(\varphi_0,u\circ\varphi_0,v\circ\varphi_0, S)$ and the associated solution $\hat\xi$  of \eqref{SEaff}. Notice that $\hat x:=\hat\xi\circ\sigma$ coincides with the usual  Carath\'eodory solution $x[\xb,\ub,u,v]$   of \eqref{E}-\eqref{EO} on $[0,T[$. Indeed, for every $t<T$, 
$$
\hat x(t)=\xi(\sigma(t))=\xb+\int_0^{\sigma(t)} [g_0(\hat \xi ,u\circ\varphi_0 ,v\circ\varphi_0 ) \varphi_0'(s)+ \sum_{i=1}^m {  g}_i(\hat\xi ,u\circ\varphi_0 ) (u_i\circ\varphi_0)'(s) ]\,ds,
$$
  by the change of variables $t=\varphi_0(s)$, we get
$$ 
\hat x(t)=\xb+ \int_0^{t}[g_0(\hat x ,u ,v)+ \sum_{i=1}^m {  g}_i(\hat x ,u) u_i'(t)]\,dt,
$$
and    Gronwall's Lemma easily  implies that $\hat x(t)=x[\xb,\ub,u,v](t)$.

\noindent Let us now show that $x=\hat x$ on $[0,T[$. By \eqref{xiv1} in Proposition \ref{xiv}, it is not restrictive to assume that in  $(\varphi_0,\varphi,\psi, S)$,  $\psi=v\circ\varphi_0$. Let ${\mathcal T}\subset[0,T[$ be the (countable)  set of discontinuity points of  $\sigma$.  Let us assume that ${\mathcal T}$ is an infinite set, the proof for ${\mathcal T}$ finite being similar, and actually simpler. For every $\tau_j\in{\mathcal T}$, set  
$s_{1,j}:=\lim_{\tau\to \tau_j^-}\sigma(\tau)$ and $s_{2,j}:=\lim_{\tau\to \tau_j^+}\sigma(\tau)$.  Clearly, $s_{1,j}<s_{2,j}< S$. Since $x$ and $u$ (as $\xi$ and $\varphi$) are continuous functions, by the definition of graph completion solution it follows that
$$
(\xi,\varphi)( s_{1,j})=(x,u)(\tau_j)=(\xi,\varphi)( s_{2,j}) \quad \text{for every $j$.}
$$
Let us set $\varphi_1:=u\circ\varphi_0$ on $[s_{1,1},s_{2,1}]$ and $\varphi_1:=\varphi$ otherwise and, for every $j\ge 1$, let us define
$\varphi_{j+1}:=u\circ\varphi_0$  in  $[s_{1,j+1},s_{2,j+1}]$ and $\varphi_{j+1}:=\varphi_j$ otherwise.  Let us consider the space-time control  $(\varphi_0,\varphi_j,v\circ\varphi_0, S)$  and let  $\xi_j$ denote the associated solution of   \eqref{SEaff}. 
 
\noindent For $j=1$,    $\xi_1(s)=\xi(s)$ for every $s\in[0, s_{1,1}]$ by definition,  moreover, 
$\xi_1(s_{1,2})=\xi_1(s_{1,1})$ since $(\varphi'_0,\varphi'_1)=(0,0)$ a.e. on $[s_{1,1},s_{2,1}]$, so that 
$$
\xi(s_{2,1})=\xi(s_{1,1})=\xi_1(s_{1,1})=\xi_1(s_{2,1}).
$$
At this point, $\xi(s)=\xi_1(s)$ also for $s>s_{2,1}$, since $\xi$ and $\xi_1$  solve  on $[ s_{2,1}, S[$ the same ODE with the same initial condition. Thus the graph completion solution $x_1:=\xi_1\circ\sigma$  coincides with the function $x$ on $[0,T[$. 
 Given $j\ge 1$, let us assume that $x_j=x$ on $[0,T[$.  Then by the same arguments it follows that  $x_{j+1}=x_j=x$ on $[0,T[$ and,  by induction, this proves that $x_j=x$ on $[0,T[$ for every  $j$.  
 
 \noindent For any $t<T$,  let ${\mathcal T}'$ be the subset of discontinuity points of $\sigma$ contained on $[0,t]$ and set $\bar S:=\sigma(t^+)$ \, ($<S$). 
 By definition,  $(\varphi_j)_j$ pointwisely converges to $u\circ\varphi_0$.    In order to prove that  the sequence  $(\varphi_j)_j$  converges uniformly in  $[0,\bar S]$  (to  $u\circ\varphi_0$), let us define, for every $j$, $\tilde\varphi_j$ as $\tilde\varphi_j:=\varphi_j$ on $[0, \bar S]$ and  $\tilde\varphi_j:=\varphi$ on $]\bar S,S[$. 
Then, 
for every $k$ and $j$ with $k>j$,
 $$
 \begin{array}{l}
 \sup_{s\in[0,\bar S]}|\varphi_k(s)-\varphi_j(s)|= \sup_{s\in[0,S[}|\tilde\varphi_k(s)-\tilde\varphi_j(s)| 
 \le \\ \, \\
\qquad   \sum_{i=j}^{k-1}\int_{ s_{1,i+1}}^{s_{2,i+1}}|\tilde\varphi'_{i}(s)|\,ds\le  \underset{i=j,\dots, k-1, \ \tau_{i+1}\in {\mathcal T}'}{\sum} (s_{2,i+1}-s_{1,i+1}),
   \end{array}
  $$
  where the last expression  tends to zero as $j\to+\infty$ since 
  $$
  \underset{i=1,\dots,\infty, \  \ \tau_i\in {\mathcal T}'}{\sum}(s_{2,i}-s_{1,i})\le \bar S<+\infty. 
  $$
Hence , in view of Proposition \ref{xiv},  $(\xi_j)_j$ converges uniformly to $\hat\xi$ on $[0,\bar S]$ and we get
$$
x(\tau)=\lim_j x_j(\tau)=\lim_j\xi_j(\sigma(\tau))=\hat\xi(\sigma(\tau))=\hat x(\tau) \qquad \forall \tau\in[0,t].
$$
By the arbitrariness of $t<T$,  this implies (i), namely the  equality $x=\hat x$ on $[0,T[$.

\vsm 

 If $S<+\infty$ statement (ii) holds true, since  $x=\hat x$ holds  on $[0,T]$. When $S=+\infty$, by definition,  (ii) is verified if and only if  $(x,u)(T)\in (x,u)_{set}(T)$, being $(x,u)_{set}(T)=(\hat x,u)_{set}(T)$  in view of (i).  To conclude the proof it remains to show that,  if
  $(\varphi_0,\varphi,\psi, S)$ is a   graph completion with BV loops of $(u,v)$ with $S=+\infty$, then $(x,u)(T)\in (x,u)_{set}(T)$.  By  \eqref{len} it follows that
\bel{sconv}
 Var_{\R^+}(\varphi_{|\varphi_0'=0})=\sum_{j=1}^{+\infty}(s_{2,j}-s_{1,j})<+\infty.
\eeq
Let $(s_i)_i$  be an increasing, diverging sequence such that  $\lim_i(\xi,\varphi)(s_i)=(x,u)(T)$, existing in view of Definition \ref{GClocS}.  For every $i$, set  $t_i:=\varphi_0(s_i)$. If  there is some subsequence of $(s_i)_i$, which we still denote   by $(s_i)_i$,  such that every $t_i$  does not belong to  ${\mathcal T}$,   we have $t_i\nearrow T$ and   we get 
\bel{vT}
\lim_i(x,u)(t_i)=\lim_i(\xi,\varphi)(s_i)=(x,u)(T).
\eeq
By Definition \ref{DACloc}, this implies that  $(x,u)(T)\in (x,u)_{set}(T)$. Otherwise, possibly disregarding a finite number of terms, we can suppose that $(t)_i\subset {\mathcal T}$. In this case,  $\varphi_0$ is constant on an interval  where $(\xi,\varphi)$ describes a loop. Precisely, if $t_i$ coincides with the element $\tau_{j}\in  {\mathcal T}$,  
$$
 \varphi_0(s)=t_i     \ \  \text{for all $s\in[s_{1,j},s_{2,j}]$,} \qquad
 (\xi,\varphi)(s_{2,j})= (\xi,\varphi)(s_{1,j})=(x,u)(t_i).
 $$
By the last equality, if there is some subsequence of $(s_i)_i$ such that every $s_i$ coincides with either $s_{1,j}$ or $s_{2,j}$ for some $j$, we get \eqref{vT} and we can conclude as above. If instead, possibly disregarding a finite number of terms,  $s_i\in]   s_{1,j},s_{2,j}[$ for every $i$, recalling that $(\xi,\varphi)$  is bounded, we obtain by standard estimates  that $(\xi,\varphi)$ is Lipschitz continuous, so that
$$
|(\xi,\varphi)(s_i)-(x,u)(t_i)|\le \sup_{s\in[s_{1,j},s_{2,j}]}|(\xi,\varphi)(s)-(\xi,\varphi)( s_{1,j})|\le C(s_{2,j}-s_{1,j}),
$$
for some $C>0$. At this point, by \eqref{sconv} it easily follows that \eqref{vT} still holds and the  proof of (iii)  is concluded. \end{proof}

\vv
\section{BV$_{loc}$ simple limit solutions}\label{SLS}
Let us begin recalling the  notion  of simple  and  of BV simple limit solution,   given in \cite{AR} for vector fields $g_1,\dots, g_m$ depending on $x$ only and extended to $(x,u)$-dependent data in \cite{AMR} \footnote{In  \cite{AR},   \cite{AMR}  also more general,   not necessarily simple,  limit solutions have been defined.}. We use $\L(T):=\L([0,T], U)$ to denote the set of pointwisely defined,  Lebesgue integrable inputs. 
\begin{Definition}[S and BVS limit solution]\label{edsdef}   
Let   $(u,v)\in \L(T)\times L^1(T)$ with $u(0)=\ub$. 
 A map $x$ is called a  {\em simple limit solution,} shortly   {\em $S$  limit solution}, of \eqref{E}-\eqref{EO},  if   there exists a sequence of controls   $(u_k)_k\subset  AC(T)$  with $u_k(0)=\ub$, pointwisely  converging to $u$  and such that
\begin{itemize}
\item[{(i)}] the sequence  $(x_k)_k$ of  the  Carath\'eodory solutions to  \eqref{E}-\eqref{EO} corresponding to $(u_k,v)$  is equibounded on $[0,T]$; \,
\item[{(ii)}]  for any $t\in[0,T]$, \     $\lim_k x_k(t)=x(t).$ 
\end{itemize}
We say that an $S$  limit solution  $x$ is a  {\rm BV simple limit  solution}, shortly  a {\rm BVS  limit  solution}, of  \eqref{E}-\eqref{EO}   if   the approximating inputs  $u_k$ have equibounded variation. 
 
\end{Definition}

Let us   introduce the  new definition of BV$_{loc}S$ limit solution.
\begin{Definition}[BV$_{loc}S$ limit solution]\label{edsdefnew} 
Let   $(u,v)\in \L(T)\times L^1(T)$ with $u(0)=\ub$.  We say that an $S$  limit solution  $x$ is a  {\em BV$_{loc}$ simple limit  solution,}  shortly  a {\em BV$_{loc}$S limit  solution}, of \eqref{E}-\eqref{EO}:  
\begin{itemize}
\item[{(i)}]  {\em  on $[0,T[$},  if there exist a sequence of controls  $(u_k)_k$ as in the  definition of $S$  limit solution,   such that
 for any $t\in]0,T[$  the approximating inputs  $u_k$ have equibounded variation on $[0,t]$; 
 \item[{(ii)}]  {\em  on $[0,T]$},  if, moreover, $x$   is bounded and  there exist  a positive, decreasing map   $\tilde\varepsilon$ with $\lim_{s\to+\infty}\tilde\varepsilon(s)=0$ and   two strictly increasing, diverging sequences $(\tilde s_j)_j\subset\R_+$,   $(k_j)_j\subset\N$, $k_j\ge j$,  such that,  for every $k>k_j$: 
\bel{cBVl}
\exists\, \tau^j_k <T: \quad  \tau^j_k+ Var_{[0,\tau^j_k]}(u_k)=\tilde s_j,  \quad |(x_k,u_k)(\tau^j_k)- (x_k,u_k)(T)|\le\tilde\varepsilon(j).  
\eeq 
\end{itemize}
\end{Definition}

\begin{Remark} {\rm
By Definition \ref{edsdef}  it follows  that, if $x$ is a BVS limit solution associated to $(u,v)$, then $u\in BV(T)$. 
Analogously,  when $x$ is a BV$_{loc}$S limit solution corresponding to $(u,v)$, Definition  \ref{edsdefnew}  implies that $u\in \ol{BV}_{loc}(T)$.}
\end{Remark}

\begin{Remark}\label{prec} 
{\rm  The S,  BVS, and BV$_{loc}$S  limit solution associated to a control $(u,v)$  is not unique, unless   the system is commutative.  Clearly,  any BVS limit solution is a BV$_{loc}$S  limit solution, which is an S limit solution, so that the sets of  S,  BV$_{loc}$S and BVS  limit solutions form a decreasing sequence of sets.}
\end{Remark}

\begin{Remark} {\rm  Following \cite{AR}, in the above definition  the approximating regular trajectories $x_k=x[\xb,\ub,u_k,v]$ are obtained keeping  the ordinary control $v$ fixed. This is in fact equivalent to consider approximating solution $x_k=x[\xb,\ub,u_k,v_k]$, where $v_k\to v$ in $L^1$-norm (see \cite{MS3}).}
\end{Remark}
  \begin{Remark}\label{condST} {\rm As we will see in Theorem  \ref{TST2}, 
  condition (\ref{cBVl}) guarantees that  a BV$_{loc}$S  limit solution $x$  is a BV$_{loc}$ graph completion solution on $[0,T]$, not only on $[0,T[$.   Actually, we will prove that   any   $x$   verifying  part (i) of Definition \ref{edsdefnew}   turns out to be a BV$_{loc}$ graph completion solution on $[0,T[$.

\noindent Condition  (\ref{cBVl}) is more meaningful once we read it as an hypothesis on the  the graphs of the approximating sequence $(x_k,u_k)_k$. Precisely, 
for any trajectory-control pair $(x_k,u_k,v)$ as in  Definition \ref{edsdefnew}, let $(\xi, \varphi_{0_k},\varphi_k, v\circ \varphi_{0_k}, S_k)$ be its arc-length graph parametrization (see Definition \ref{Dal}).  Then    (\ref{cBVl})   is equivalent to: 

\noindent {\it the existence of   a positive, decreasing map  $\tilde\varepsilon$ with $\lim_{s\to+\infty}\tilde\varepsilon(s)=0$ and  of two strictly increasing, diverging sequences $(\tilde s_j)_j\subset\R_+$  and  $(k_j)_j\subset\N$, $k_j\ge j$,  such that,  for every $k>k_j$: 
\bel{cBVe}
|(\xi_k,\varphi_k)(\tilde s_j)- (\xi_k,\varphi_k)(S_k)|\le\tilde\varepsilon(j).
\eeq}
Clearly, (\ref{cBVe}) holds true when the sequence $(\xi_k,\varphi_k)_k$ is uniformly convergent on $\R_+$ (by considering, for every $k$,  the  extension  $(\xi_k,\varphi_k)(s):=(\xi_k,\varphi_k)(S_k)$ for every $s\ge S_k$).}
\end{Remark}

As an immediate  consequence of Theorems \ref{Egc} and \ref{TST2}, we have the following existence result for BV$_{loc}$S limit solutions.
\begin{Corollary}  If $U$ has the Whitney property, then for  any $(u,v)\in \ol{BV}_{loc}(T)\times L^1(T)$ with $u(0)=\ub$   there exists an associated BV$_{loc}$S  limit solution $x$ to  \eqref{E}-\eqref{EO} on $[0,T[$ (on  $[0,T]$,  when $x$ is bounded).  \end{Corollary}

As a by-product, we get that  every function  $u\in \ol{BV}_{loc}(T)$  is the  pointwise limit on $[0,T]$ of a sequence  $(u_k) \subset AC(T)$  with {\it equibounded variation} on every interval $[0,t]$ with $t<T$  and verifying  (\ref{cBVl}). 

\vv

 Let us conclude this section with an example, illustrating  the relations between the notions of  AC$_{loc}$  solutions,  BV$_{loc}$ graph completion solutions and of BV$_{loc}$S  limit solutions   considered in  Definitions \ref{DACloc},  \ref{GClocS}, and \ref{edsdefnew} above.

 \begin{Example}\label{EIntr} 
  {\rm   Let us consider  the control system in $\R^3$
\bel{es3}
\dot x =  g_1(x)\dot u_1+g_2(x)\dot u_2,  \qquad |u|\le 1 
\eeq
with $u\in \R^2$ and initial conditions
\bel{ic1}
x(0)=(1,0,1),\quad u(0)= (1,0),
\eeq
where  
$$
g_1(x){:=}\eta(x) \begin{pmatrix} 1\\0\\x_3x_2 \end{pmatrix},\quad 
g_2(x){:=} \eta(x) \begin{pmatrix} 0\\1\\-x_3x_1 \end{pmatrix},
$$
and $\eta$ is a Lipschitz continuous function equal to $1$ as  $|x|\leq 3$ and equal to $0$ as $|x|\geq 4$ \footnote{The multiplication by the {\em cut-off function}   $\eta$, while unneeded, is sufficient to guarantee the sublinearity hypothesis on the dynamics.}.
 \vv
 {\bf (i)} For any   control $u\in AC(T)$  verifying $u(0) = (1,0)$,  the corresponding  Carath\'eodory solution to (\ref{es3}), (\ref{ic1}) 
is 
$$
(x_1,x_2,x_3)(t)= \left(u_1(t),u_2(t), e^{-\int_0^t(-u_2\dot u_1 +u_1\dot u_2)(s)\,ds}\right)  \quad\forall t\in[0,T].
$$
In particular,
$ 
x_3(T)\ge e^{- Var_{[0,T]}(u)}>0.
$

Hence, if given a control $u\in BV(T)$   we consider just BVS  limit solutions    to (\ref{es3}), (\ref{ic1}), that is, pointwise limits of Carath\'eodory solutions  corresponding to  approximating  inputs $u_k$ with { \it equibounded variation} (see Definition \ref{edsdef}), 
we always obtain $x_3(T)>0$.
Similarly, if,  we introduce graph completions  $(\varphi_0,\varphi):[0,S]\to[0,T]\times U$  of $u$, with $S$ (and thus $Var_{[0,S]}(\varphi)$) finite, for any clock $\sigma$ we get a  graph completion solution with $x_3(T)>0$ (see Definitions \ref{GCloc},  \ref{GClocS}). Precisely, the space-time system is 
\bel{STes3}
(\xi_1,\xi_2,\xi_3)' =  g_1(\xi)\varphi_1'+g_2(\xi)\varphi'_2,  \quad \xi(0)=(1,0,1), \ \varphi(0)=(1,0),
\eeq
where $\varphi_0(0)=0$,  $\varphi_0(S)=T$,  $\varphi_0'\ge0$ and  $\varphi_0'+|\varphi'|=1$ a.e. on $[0,S]$,   so that,
\bel{es3xi}
 \xi_3(s)= e^{-\int_0^s(-\varphi_2 \varphi_1' +\varphi_1  \varphi_2')(s)\,ds} \quad \text{for } \,\,s\in[0,S],
\eeq 
and  $|\int_0^s(-\varphi_2 \varphi_1' +\varphi_1  \varphi_2')(s)\,ds|\le  S$. Thus  the graph completion solution, defined by $x=\xi\circ\sigma$,  
verifies $x_3(T)=
\xi_3(S)\ge e^{-S}>0$.

\vsm
Let us now consider inputs  $u\in\ol{AC}_{loc}(T)$.  In this case, if we  set, for instance,
\bel {ACLex}
u(t){:=} \left( \cos\left(\frac{1}{T-t}-\frac{1}{T}\right),\sin\left(\frac{1}{T-t}-\frac{1}{T}\right) \right)  \ \text{ for } t\in[0,T[, \ u(T)=(1,0),
\eeq
 the corresponding solution to (\ref{es3}), (\ref{ic1}) on $[0,T[$
has the third component
$x_3(t)=  e^{- \frac{t}{T(T-t)}}$,
  so that the  extension $(x_1,x_2,x_3)(T)\equiv (u_1,u_2,x_3)(T):=(1,0,0)$ gives  a feasible  AC$_{loc}$ trajectory-control pair (see Definition \ref{DACloc}). 
In fact,   such an extended  map $x$ is also    a BV$_{loc}$S limit solution (see Definition \ref{edsdefnew}). Indeed, for every $k$,  let us set  
\bel{acloc}
t_k:=\frac{2k\pi T^2}{1+2k\pi T}, \quad u_k(t):= u(t)\chi_{[0,t_k]}(t)+u(t_k)\chi_{]t_k,T]}(t), 
\eeq
where $u$ is as in (\ref{ACLex}), 
so that   $u(t_k)=(\cos(2k\pi),\sin(2k\pi))=(1,0)$. 
Then   $x$ is the
 pointwise limit of the Carath\'eodory solutions $x_k$ of (\ref{es3}), (\ref{ic1}) corresponding to 
the   controls $u_k\in AC(T)$, with
$
Var_{[0,t]}(u_k)\le \frac{t}{T(T-t)} \quad \forall t\in[0,T[
$
and  $(x _k,u_k)(T)=(x_k,u_k)(t_k)$,  so that easy calculations yield all the remaining conditions of Definition \ref{edsdefnew} below.  In particular \eqref{cBVl} is verified if we choose $\tilde s_j:=t_j+Var_{[0,t_j]}(u)$, where $t_j=\frac{2j\pi T^2}{1+2j\pi T}$,  and 
$k_j:=j$, so that if we set $\tau_k^j:=t_j,$ we get $\tau_k^j+Var_{[0,\tau_k^j]}(u_k)=t_j+Var_{[0,t_j]}(u)=\tilde s_j$ and,
for every $k\ge j$,  we have
$$|(x _k,u_k)(t_j)-(x _k,u_k)(T)|=|(x,u)(t_j)-(x,u)(t_k)|=|x(t_j)-x(t_k)|\le e^{- \frac{t}{T(T-t_j)}},$$
where the last term, independent of $k$, tends to zero as $j\to\infty.$

\vv
 {\bf (ii)} For $(x,u)$  solution of  
 (\ref{es3})-(\ref{ic1}), let us consider 
the  problem of minimizing the following payoff
 $$
 J(u):=\int_0^T[|1-u_1(t)|+|u_2(t)|+|x_3(t)||\dot u(t)| ]\,dt 
 $$
 subject to the constraints
 $$
 (x,u)(T)\in \C:=   (U\times\{0\})\times U.
 $$ 
 By {\bf (i)},  no AC trajectory-control pairs $(x,u)$ verifying the constraints exist, hence $\inf_{u\in AC(T)}J(u)=+\infty$.    In the extended class of    AC$_{loc}$  trajectory-control pairs, as observed in Remark \ref{Remmin},    the terminal constraint is equivalent to assume that
$$
 \liminf_{t\to T^-}d\big((x(t),u(t)),\, \C\big)=0.
$$
Hence, for every $k$,  implementing the control  
$$
u_k(t):=(1,0)\chi_{[0,T-(1/k)]}+  \left( \cos\left(\frac{1}{T-t}-k\right),\sin\left(\frac{1}{T-t}-k\right) \right)\chi_{[T-(1/k),T[}
$$
we get the solution 
$$
x_k(t)=(1,0,1)\chi_{[0,T-(1/k)]}+ \left(u_{1_k}(t),u_{2_k}(t), e^{k-\frac{1}{T-t}}\right)\chi_{[T-(1/k),T[},
$$
with $(x_k,u_k)$  verifying the constraints and  
$ 1\le J(u_k)\le 1 +\frac{3}{k},$ 
 so that  $\lim_{k}J(u_k)=1$. In fact, it is not difficult to prove that 1 is the infimum (but not the minimum)  cost  in the class of AC$_{loc}$ controls.  The minimum   does exist, and is equal to 1, over the   set of BV$_{loc}$ graph completions: it suffices  to consider the space-time control
 \bel{ex1f1}
 (\varphi_0,\varphi)(s):=(s,1,0)\chi_{[0,T[}(s)+(T,(\cos(s-T),\sin(s-T))\chi_{[T,+\infty[}(s)
\eeq
 and the corresponding trajectory
\bel{ex1f2}
 \xi(s)=(1,0,1)\chi_{[0,T[}(s)+(\cos(s-T),\sin(s-T),e^{-s+T})\chi_{[T,+\infty[}(s).
 \eeq
 Notice that, by adding to the system the variable
\bel{ex1f3}
 \dot x_4=|1-u_1(t)|+|u_2(t)|+|x_3(t)||\dot u(t)|, \qquad x_4(0)=0
\eeq
  in the space-time setting we can consider the {\it extended} payoff
 $$
 {\mathcal J}(\varphi_0,\varphi, S):=\int_0^S[(|1-\varphi_1(s)|+|\varphi_2(s)|)\varphi_0'(s)+|\xi_3(s)||\varphi'(s)|]\,ds,
 $$
 where $S\le+\infty$ and $\lim_{s\to S}\varphi_0(s)=T$. Hence by   \eqref{ex1f1},   \eqref{ex1f2}, we get ${\mathcal J}(\varphi_0,\varphi, +\infty)=1$.
 Finally,  in the class of $S$  limit solutions, where   the optimization problem is equivalent to  minimize $x_4(T)$,   the minimum cost  is still equal to 1. In particular,  for every sequence  $(x_k,u_k)_k$ of equibounded, absolutely continuous maps defining an  $S$  limit solution verifying the terminal constraint, one has $\lim_k Var_{[0,T]}(u_k)=+\infty$ and 
 $$
 x_{4_k}(T)=J(u_k)\ge \int_0^T e^{-\int_0^t|\dot u_k|\,dr}\,|\dot u_k|\,dt=1-e^{-Var_{[0,T]}(u_k)}\to 1 \quad\text{as $k\to+\infty$.}
 $$ 
   Actually,  in view of Theorem \ref{TST2} below,   the minimum value  is obtained in the subset of BV$_{loc}$S limit solutions (see Definition \ref{edsdefnew}).
   }
\end{Example} 

\section{Well posedness and characterization}\label{ASS}
 Our main result is the following equivalence between  BV$_{loc}$ graph completion solutions  and   BV$_{loc}$S  limit solutions.
 \begin{Theorem} \label{TST2} Let us assume that $U$ has the Whitney property. Let  $(u,v)\in \ol{BV}_{loc}(T)\times L^1(T)$ with $u(0)=\ub$. Then
 \begin{itemize}
 \item[(i)] {\rm (Well posedness)} a {\rm BV}$_{loc}$  graph completion solution $x$  to  \eqref{E}-\eqref{EO} is a {\rm BV}$_{loc}$S  limit solution;
 \item[(ii)] {\rm (Characterization)} Any {\rm BV}$_{loc}$S limit solution $x$  to  \eqref{E}-\eqref{EO} is a {\rm BV}$_{loc}$  graph completion solution.
 \end{itemize}   
 \end{Theorem}
  
Theorem \ref{TST2} says that any BV$_{loc}$ graph completion solution is an S limit solution. Vice-versa, given an S limit solution $x$, 
 it is a BV$_{loc}$ graph completion solution if and only if there exists an approximating sequence verifying condition (\ref{cBVl}). Precisely,   $x$ is always a BV$_{loc}$ graph completion solution on $[0,T[$:  (\ref{cBVl})  is needed to guarantee  the existence of  a BV$_{loc}$ graph completion solution {\it assuming the final value} $x(T)$. 
  
\vsm  
In order to prove that a BV$_{loc}$ graph completion solution  is a   BV$_{loc}$S  limit solution,  in Theorem \ref{Cl1-2}  below we extend  to possibly unbounded  maps the crucial approximation result of \cite[Theorem 5.1]{AR}. The proof is postponed to  Section \ref{proofs}. 

\begin{Theorem}\label{Cl1-2}
Let $\sigma:[0,T[\to\R_+$ be a strictly increasing map such that 
$$
\text{$\sigma(0)=0$ \qquad  and } \qquad \sigma(t_2)-\sigma(t_1)>(t_2-t_1) \quad \forall t_1,\, t_2\in[0,T[, \quad  t_1< t_2.
$$
Set
$$
 \lim_{t\to T^-}\sigma(t)=:\bar S \qquad (\bar S\le+\infty).
 $$
 Let $\varphi_0:\R_+\to[0,T[$ be the unique, (1-Lipschitz continuous) increasing, surjective map  verifying
$$
 \varphi_0\circ\sigma(t)=t \quad \forall t\in[0,T[, \quad \text{and, if $\bar S<+\infty$, } \quad \varphi_0(s):=T \quad \forall s\ge \bar S.
$$ 
 Then there exists a sequence of absolutely continuous, strictly increasing  maps $\sigma_h:[0,T[\to\R_+$ such that 
\begin{itemize}
\item[(i)]  $\sigma_h(0)=0$,   $\lim_{t\to T}\sigma_h(t)=+\infty$,   and
\bel{cp}
\lim_{h}\sigma_h(t)=\sigma(t) \quad \forall t\in[0,T[;
\eeq
\item[(ii)]  the maps $\varphi_{0_h}:=\sigma_h^{-1}:\R_+\to[0,T[$ are strictly increasing, $1$-Lipschitz continuous,  surjective, converge locally uniformly to $\varphi_0$  
and   verify, for every $t\in]0, T[$,  
\begin{itemize}
\item[(ii.1)] if $\bar S<+\infty$, setting  $\varepsilon(h):=\sup_{s\in[0,\bar S]}|\varphi_{0_h}(s)-\varphi_0(s)|$: 
 $$
\sup_{s\in\R_+}|\varphi_{0_h}(s)-\varphi_0(s)|\le \varepsilon(h)+(T-t_h)  \qquad(t_h:=\varphi_{0_h}(\bar S))
$$
and  $t_h<T$ for every $h$,  $\lim_h t_h=T$,   and  $\lim_h\varepsilon(h)=0$;
\item[(ii.2)] if $\bar S=+\infty$: for every $S\in]0,+\infty[$, setting 
 $t:=\varphi_{0}(S)$  and  $t_h:=\varphi_{0_h}(S)$, 
$$
\sup_{s\in\R_+}|\varphi_{0_h}(s)-\varphi_0(s)|\le \varepsilon_S(h)+(T-\min\{t,t_h\}), 
$$
where   $\varepsilon_S(h):=\underset{s\in[0,S]}{\sup}|\varphi_{0_h}(s)-\varphi_0(s)|$, $t,t_h<T$ for every $h$,  $|t_h-t|\le \varepsilon_S(h)$  and $\lim_h\varepsilon_S(h) =0$. 
\end{itemize} 
\end{itemize}
 \end{Theorem}
  Let us point out that, even in case $\sigma([0,T[)$ is   bounded, we introduce approximating maps $\sigma_h$     from $[0,T[$ {\it onto} $\R_+$.   This is a substantial difference from  \cite[Theorem 5.1]{AR}, where $\sigma\in L^1([a,b],[0,1])$ and every approximation  $\sigma_h$ maps   $[a,b]$ onto $[0,1]$.

\subsection{Proof of Theorem \ref{TST2}: Well posedness}\label{dim1}

   Let us begin by showing that a   BV$_{loc}$ graph completion solution is a BV$_{loc}$S limit solution.
 We limit ourselves to consider just BV$_{loc}$ graph completions which are not BV, since this last case was already covered by  \cite[Theorem 4.2]{AR}.  Let $x$ be a BV$_{loc}$  graph completion solution to  \eqref{E}-\eqref{EO},  which, by Definitions \ref{fsttc} and \ref{GClocS},  is  associated to a feasible  space-time trajectory-control pair  $(\varphi_0,\varphi,\psi, +\infty)\in \U(T;\ub)$,  $\xi:=\xi[\xb,\ub, \varphi_0,\varphi,\psi]$  with  $\xi$ bounded,  and to a strictly increasing function $\sigma:[0,T[\to\R_+$, such that: 
\bel{fsj}
\left\{\begin{array}{l}
(\xi,\varphi_0,\varphi,\psi)\circ\sigma(t)=(x(t),t, u(t),v(t))   \qquad \forall t\in[0,T[, \\ \, \\
\lim_j(\xi(s_j),\varphi(s_j))=(x(T), u(T)) \quad \text{for some \ }s_j\nearrow +\infty.
\end{array}\right.
\eeq
Let
\bel{bS}
\bar S:=\inf\{s>0: \ \varphi_0(s)=T\}.
\eeq 
We consider separately the two cases $\bar S=+\infty$ and $\bar S<+\infty$, since they require a  different construction  of  the  equibounded,   approximating sequence $(x_k,u_k)_k$ of $(x,u)$. Precisely we will prove the following
\vsm
\noindent{\sc \underline{Claim}:} 
There exists a sequence $(u_k)_k\subset AC(T)$, $u_k(0)=\ub$, and $x_k:=x[\xb,\ub, u_k,v]$ verifying these properties:
\begin{itemize}
\item[(i)] for every $t\in[0,T]$, 
\bel{Thst}
|(x_k,u_k)(t)- (x,u)(t)|+ \|(x_k,u_k)- (x,u)\|_{L^1(T)}\to 0 \ \ \text{as} \ \ k\to+\infty,
\eeq
\item[(ii)]  there exists an increasing function $V:[0,T[\to\R_+$ with $V(0)=0$ and $\lim_{t\to T}V(t)=+\infty$, such that, for every $k$, 
\bel{Var}
Var_{[0,t]}(u_k)\le V(t) \qquad \text{for every $t\in]0,T[$;}
\eeq
\item[(iii)] in correspondence to the sequence $(s_j)_j$ in \eqref{fsj},  there exist     a positive, decreasing sequence  $\tilde\varepsilon$ with $\lim_j\tilde\varepsilon(j)=0$ and  a strictly increasing sequence $(k_j)_j$ $(k_j\ge j)$ such that, defining implicitly   $ \tau^j_k$    by
$$
\tau^j_k+ Var_{[0,\tau^j_k]}(u_k)=s_j,
$$
 one has
$$
|(x_k,u_k)(\tau^j_k)- (x_k,u_k)(T)|\le\tilde\varepsilon(j),
$$ 
\end{itemize}
\vsm
\noindent  so  that $(x,u,v)$ is a BV$_{loc}$S  limit solution on $[0,T]$.   
 
 \noindent In both cases, as a first step,  using Theorem \ref{Cl1-2}   we define a sequence of strictly increasing, Lipschitzean  maps $\varphi_{0_h}$ approaching locally uniformly $\varphi_0$ as $h\to\infty$ and  consider the trajectory-control pairs   $(\xi,\varphi)\circ(\varphi_{0_h})^{-1}$.  
Furthermore, we obtain an equibounded subsequence belonging to  ${AC}(T)$  by truncating and then carefully modifying  the (non BV) controls  $\varphi\circ\varphi_{0_h}^{-1}$,  using the property (\ref{hCA}).   
 Notice that
 $$
 \lim_{t\to T^-}\sigma(t)=\bar S.
 $$
In particular,  when $\bar S<+\infty$  the pair $(x,u)$ has a jump at the final time $t=T$ from $(x,u)(T^-)=(\xi,\varphi) (\bar S)$ to $(x,u)(T)$ and Var$_{[\bar S, +\infty[}(\varphi)=+\infty$. 

\vsm 
\noindent{\sc \underline{Case 1}:} let $\bar S<+\infty$. In view of Theorem \ref{Cl1-2},    there exists a sequence of absolutely continuous, strictly increasing functions  $\sigma_h$  from $[0,T[$ onto $ \R_+$ and pointwisely converging to $\sigma$ such that, for every $h$, 
  the maps $\varphi_{0_h}:=\sigma_h^{-1}:\R_+\to[0,T[$ are strictly increasing, $1$-Lipschitz continuous,  surjective 
and they verify, for every $h$, 
\bel{t1}
\sup_{s\in\R_+}|\varphi_{0_h}(s)-\varphi_0(s)|\le \varepsilon(h)+(T-t_h), 
\eeq
where $t_h:=\varphi_{0_h}(\bar S)\nearrow T$, $\varepsilon(h):=\sup_{s\in[0,\bar S]}|\varphi_{0_h}(s)-\varphi_0(s)|$ and  $\lim_h\varepsilon(h)=0$.
\newline
Let us  define
$$(x,u,v):=  (\xi,\varphi,\psi)\circ\sigma,
$$
$$
u_h:=\varphi\circ\sigma_h,\,\, x_h:=x[\xb,\ub,u_h,v],\,\,
\xi_h:= \xi[\xb,\ub,\varphi_{0_h}, \varphi,v\circ\varphi_{0_h}].
$$
Clearly,  $x_h=\xi_h\circ\sigma_h$ on $[0,T[.$   Let $(s_j)_j$ be as in \eqref{fsj} and for every $j$, $h$, let us set
\bel{t2}
 \tau^j_h:=\varphi_{0_h}(s_j).
\eeq
Since $s_j\nearrow+\infty$, it is not restrictive to assume $s_j>\bar S$ for every $j$; hence,  for every $h$ the sequence $(\tau^j_h)_j$ is strictly increasing   and, for every $j$, 
$$
t_h\le \tau^j_h<T  \quad\text{and}\quad \lim_h \tau^j_h=\lim_h t_h=T.
$$
In order to construct an equibounded trajectory-control sequence   verifying   (\ref{Thst}) and (\ref{Var}),  let us preliminary notice that, for every $j$, by Proposition \ref{xiv}  we have, for any $h$,
\bel{xij}
\sup_{s\in[0,s_j]}|\xi_h(s)-\xi(s)|=: \varepsilon_j^1(h),
\eeq
 with $\varepsilon_j^1(h)\le \varepsilon_{j+1}^1(h)$ and $\lim_h \varepsilon_j^1(h)=0$.  We define  a  sequence 
  $(h_j)_j$  as follows. Choose  $h_1\ge 1$ verifying 
 $\varepsilon^1_1(h)\le 1$ for every $h\ge h_1$  and
for  any $j>1$  let  $h_j>h_{j-1}$ ($\ge j-1$) be such that 
\bel{xijfrac}
\varepsilon(h), \ \varepsilon^1_j(h)\le \frac{1}{j} \quad\text{ for every $h\ge h_j.$}
\eeq
Using  the Whitney property (\ref{hCA}),   let us set
\bel{sAC}
\begin{array}{l}
u_j:=u_{h_j }\chi_{[0,\tau_{h_j}^j]}+\tilde u_j\left(\frac{t-\tau_{h_j}^j}{T-\tau_{h_j}^j}\right)\chi_{]\tau_{h_j}^j,T]},  \\ \, \\
x_j:=x[\xb,\ub,u_j,v],
\end{array}
\eeq
 where  $\tilde u_j\in AC(1)$ joins  $u^j_{h_j}(\tau_{h_j}^j)=\varphi(s_j)$ to  $u(T)$ and $Var_{[0,1]}(\tilde u_j)\le C| \varphi(s_j)- u(T)|$.  
 Since  $u_j(T)=u(T)$ for every $j$, $\lim_j u_j(T)=u(T)$ trivially. If $t\in[0,T[$,  there is some $j$ such that $t\le   \tau_{h_j}^j$ and we have 
\bel{luh}
\lim_{j}u_{h_j}(t)=\lim_{h_j}\varphi(\sigma_{h_j}(t))=\varphi(\sigma(t))=u(t),
\eeq
recalling that  $\varphi$ is a ($1$-Lipschitz) continuous function.  Moreover,  
$$
\begin{array}{l}
|x_j(t)-x(t)|=|\xi_{h_j}(\sigma_{h_j}(t))-\xi(\sigma(t))|\le \\ \, \\
\qquad\qquad \qquad |\xi_{h_j}(\sigma_{h_j}(t))-\xi(\sigma_{h_j}(t))| + |\xi(\sigma_{h_j}(t))-\xi(\sigma(t))|, 
\end{array}
$$
where  $\sigma_{h_j}(t)\in[0,s_j]$ and
$$
 |\xi_{h_j}(\sigma_{h_j}(t))-\xi(\sigma_{h_j}(t))| \le\sup_{s\in[0,s_j]}|\xi_{h_j}(s)-\xi(s)|\le 1/j.
 $$
 Since $\xi$ is continuous, this implies that  $\lim_j x_j(t)=x(t)$ for every $t\in[0,T[$.  

\noindent Let $t\in[0,T[$. To prove the existence of a function $V$ such that (\ref{Var}) holds true, 
notice that   $\lim_j\sigma_{h_j}(t)=\sigma(t) <+\infty$ (actually, $\sigma(t)\le \bar S$). Therefore,  $\sigma_{h_j}(t)\le \sigma(t)+1$ for every $j>   j(t)$ for some  integer $j(t)$  and 
$$
\sigma_{h_j}(t)\le \sigma(t)+M(t), \quad\text{if \ } M(t):=\max\{1, \max\{\sigma_{h_j}(t): j=1,\dots,  j(t)\}\}<+\infty.
$$
  By the above estimate, for any $j$ such that $t\le \tau_{h_j}^j$,  we get
$$
 Var_{[0,t]}(u_j)=  Var_{[0, \sigma_{h_j}(t)]}(\varphi)\le Var_{[0, \sigma(t)+M(t)]}(\varphi),
 $$
while if $t>\tau_{h_j}^j$, 
$$
\begin{array}{l}
 Var_{[0,t]}(u_j)\le  Var_{[0,t]}(u_{h_j})+\int_{ \tau_{h_j}^j}^t|\dot u_j(t)|\,dt \le \\ \, \\
   Var_{[0,\sigma(t)+M(t)]}(\varphi)+C\,\text{diam}(U).
 \end{array}
 $$
 Therefore  $(u_j)_j$ verifies (\ref{Var}) if we choose
 \bel{TV}
 V(t):=Var_{[0,\sigma(t)+M(t)]}(\varphi)+C\,\text{diam}(U) \quad \forall t\in[0,T[.
\eeq
Let us now prove that the sequence $(x_j)_j$ is equibounded. In view of the boundedness of $\xi$ and $x$ and of the previous estimates,  we get
 $$
  \sup_{t\in[0,  \tau_{h_j}^j]}|x_j(t)|\le \sup_{t\in[0,T[}|x(t)|+2\sup_{s\ge 0}|\xi(s)|+(1/j)\le \sup_{t\in[0,T[}|x(t)|+2\sup_{s\ge 0}|\xi(s)|+1=:R.
 $$
If instead $t>   \tau_{h_j}^j$,  by standard estimates, we have
 $$
 \begin{array}{l}
  |x_j(t)|\le \left\{|x_j(\tau_{h_j}^j)| +\right. \\ \, \\
   \left.(m+1) M[T -\tau_{h_j}^j+C\,|\varphi(s_j)-u(T)|]\right\}\, e^{(m+1) M[T-\tau_{h_j}^j+C|\varphi(s_j)-u(T)|]}\le  \\ \, \\
 \qquad\qquad\qquad\left\{R+(m+1) M[T +C\,\text{diam}(U)]\right\}\, e^{(m+1) M[T +C\,\text{diam}(U)]}=:R'.
   \end{array}
  $$
Hence, for every $j$,   
\bel{xj}
 \sup_{t\in[0,T]}|x_j(t)|\le R'.
 \eeq 
As a consequence,  by the Dominated Convergence  Theorem we also have that  
$\lim_j \|(x_j,u_j)- (x,u)\|_{L^1(T)}\to 0$. 
 	  
\noindent Finally,  for every $j$, recalling that $x_j(\tau_{h_j}^j)=\xi_{h_j}(s_j)$, we have
$$
| x_j(T)- x(T)|\le |x_j(T)-  x_j(\tau_{h_j}^j)|+ |\xi_{h_j}(s_j)- \xi(s_j)|+ |\xi(s_j)-x(T)|  
$$
where $\lim_j\xi(s_j)=x(T)$ and $|\xi_{h_j}(s_j)- \xi(s_j)|\le1/j\to0$ as $j\to+\infty$.     Using (\ref{xj}) together with standard estimates,  we get
\bel{stxT}
\begin{array}{l}
|x_j(T)-x_j(\tau_{h_j}^j)|=\left|\int_{\tau_{h_j}^j}^T\left[g_0(x_j,u_j,v)+\sum_{i=1}^mg_i(x_j,u_j)\dot u_j\right]\,dt\right| \le  \\ \, \\
\qquad(m+1)(1+ R')M[T-\tau_{h_j}^j+C|u_{h_j}(\tau_{h_j}^j)-u(T)|]\le
 \\ \, \\
\qquad\qquad\qquad\qquad  (m+1)(1+ R')M[T-t_j+C|\varphi(s_j)-u(T)|],
\end{array}
\eeq
recalling that $h_j\ge j$ so that $t_{h_j}\ge t_j$ and  hence $\lim_j|x_j(T)-x_j(\tau_{h_j}^j)|=0$. Thus  $\lim_j x_j(T)=x(T)$ and if we rename the index $j$ in the sequence $(x_j,u_j)_j$ by $k$, we obtain a sequence  $(x_k,u_k)$ verifying theses (i) and (ii). 

For every $k$, let    $(\hat\xi_k, \hat\varphi_{0_k},\hat\varphi_k, v\circ \hat\varphi_{0_k}, S_k)$ be the arc-length graph parametrization   of $(x_k,u_k,v)$ (see Definition \ref{Dal}).  In view of Remark \ref{condST}, in order to prove  (iii) we need to estimate $ |(\hat\xi_k,\hat\varphi_k)(s_j)-(\hat\xi_k,\hat\varphi_k)(S_k)|$. By the definition of $(x_k,u_k)$,  it follows that
$$
(\hat\xi_k,\hat\varphi_k)(S_k)=(x_k,u_k)(T)=(x_k,u)(T).  
$$
Moreover,  for every $k\ge j$,  we have $(\hat\varphi_{0_k}, \hat\varphi_k,\hat\xi_k) =(\varphi_{0_{h_k}},\varphi,\xi_{h_k})$ on $[0,s_j]$   (where $\varphi_{0_{h_k}}$, $\xi_{h_k}$ are the maps introduced above, with $j$ replaced by $k$) and, by \eqref{xij}, \eqref{xijfrac}, 
$$
\sup_{s\in[0,s_j]}|\hat\xi_k(s)-\xi(s)|\le \frac{1}{j}.
 $$
  Hence   for every $k\ge j$,   we  get   
 $$
 \lim_j|\hat\varphi_k(s_j) -\hat\varphi_k(S_k)|=\lim_j|\varphi(s_j)-u(T)|= 0
 $$
  independently of $k$,   and 
$$
\begin{array}{l}
 |\hat\xi_k(s_j)-\hat\xi_k(S_k)|\le |\hat\xi_k(s_j)-\xi(s_j)|+ 
 |\xi(s_j)-x(T)|+|x(T) -x_k(T)|= \tilde\varepsilon(j),  \end{array}
$$ 
where  $\lim_j\tilde\varepsilon(j)\to0$ and $\tilde\varepsilon$ does not depend on $k$, since  $|\hat\xi_k(s_j)-\xi(s_j)|\le1/j$, $|\xi(s_j)-x(T)|\to 0$   by hypothesis  and $\lim_j|x(T) -x_k(T)| =0$, being $k\ge j$. The proof of the theorem in Case 1 is thus concluded.
 \vsm
\noindent{\sc \underline{Case 2}:} let $\bar S=+\infty$.  Let  $(\sigma_h)_h$ be the sequence of absolutely continuous, strictly increasing functions   from $[0,T[$ onto $ \R_+$, pointwisely converging to $\sigma$,  whose existence is guaranteed by Theorem \ref{Cl1-2}. Let $\varphi_{0_h}:=\sigma_h^{-1}$ be the sequence of the $1$-Lipschitz continuous  inverse maps, uniformly converging  to $\varphi_0$ on any compact interval. 
  Let $(s_j)_j$ be as in \eqref{fsj}. For every $j$ and $h$,  we set 
$$
 \tau^j:= \varphi_0(s_j), \quad  \tau^j_h:= \varphi_{0_h}(s_j).
$$
Since $\varphi_0(s)<T$ for all $s\ge0$ and $\lim_{s\to+\infty}\varphi_0(s)=T$,  one has $\tau_j<T$ for every $j$  and  $\lim_j\tau^j=T$.  Passing eventually to a subsequence, it is not restrictive to assume that  the sequence $(\tau^j)_j$ is strictly increasing. Clearly, for every $h$ the sequence $(\tau^j_h)_j$ is strictly increasing, $0< \tau^j_h<T$ and $\lim_j\tau^j_h=T$.
 In view of  Theorem \ref{Cl1-2},  for every $j$ and $h$, if we set $\varepsilon_j(h):=\sup_{s\in[0,s_j]}|\varphi_{0_h}(s)-\varphi_0(s)|$,   
we have that
\bel{t1}
\sup_{s\in\R_+}|\varphi_{0_h}(s)-\varphi_0(s)|\le \varepsilon_j(h)+(T-\min\{\tau^j, \tau^j_h\}), 
\eeq
where  $\tau^j-\varepsilon_j(h)\le\tau^j_h<T$,  $\lim_h\varepsilon_j(h)=0$.  

\vv
Let us  set $(x,u,v):=  (\xi,\varphi,\psi)\circ\sigma$,  $u_h:=\varphi\circ\sigma_h$,   $x_h:=x[\xb,\ub,u_h,v]$ and $\xi_h:= \xi[\xb,\ub,\varphi_{0_h}, \varphi,v\circ\varphi_{0_h}]$, so that  $x_h=\xi_h\circ\sigma_h$.  Then, by Proposition \ref{xiv} we have
$$
\sup_{s\in[0,s_j]}|\xi_h(s)-\xi(s)|=:\varepsilon_j^1(h),
$$
where   $\lim_h \varepsilon_j^1(h)=0$. Now, similarly to Case 1, let us introduce a sequence $(h_j)_j$  such that 
$$
\varepsilon_j(h), \ \varepsilon^1_j(h)\le \frac{1}{j} \quad\text{ for every $h\ge h_j$.}
$$
At this point, the sequence of absolutely continuous functions $(x_j,u_j)_j$ defined  as in (\ref{sAC}) is equibounded and  converges pointwisely and in $L^1$-norm to $(x,u)$. Indeed, it is enough to observe that $\tau^j-(1/j)\le\tau^j_{h_j}<T$, so that $\lim_j  \tau^j_{h_j}=T$ and afterwards the proof is the same as in Case 1.  \qed
 
 \vv
\subsection{Proof of Theorem \ref{TST2}: Characterization}
 
 Let us now prove that  a BV$_{loc}$S  limit solution $x$ is a BV$_{loc}$  graph completion solution.  
Let us assume the  {\sc \underline{Claim}}   at the beginning of
Subsection \ref{dim1}
as hypothesis.

 For every $k$, set $V_k: =Var_{[0,T]}(u_k)$ \,  ($<+\infty$). Taking eventually a subsequence,  we can assume that the sequence  $(V_k)_k$ of the  variations  is increasing.  If  this sequence is bounded,  $x$ is in fact a BVS  limit solution and it coincides with a BV graph completion solution by \cite[Theorem 4.2]{AR}.  Hence let us assume  
\bel{Vinfty}
\lim_k V_k=+\infty.
\eeq 
In order to prove that $x$ is a BV$_{loc}$  graph completion solution on $[0,T]$, let us consider the arc-length graph parametrizations of the inputs $u_k$. Precisely, let us  define for every $k$,   a map $\sigma_k:[0,T]\to[0,T+V_k]$ by setting
\bel{sk}
\sigma_k(t):= t+Var_{[0,t]}(u_k) \quad (\le t+V(t))
\eeq 
and let $\varphi_{0_k}:\R_+\to[0,T]$ be   the $1$-Lipschitz continuous,  increasing function such that  
$$
\varphi_{0_k}:=\sigma_k^{-1} \ \text{ on $[0,T+V_k]$, and }  \ \varphi_{0_k}(s)=T \ \text{ for all $s\ge T+V_k$.}$$
Set $\varphi_k:= u_k\circ\varphi_{0_k}$. Then the sequence of space-time controls  $(\varphi_{0_k},\varphi_k)_k$   is $1$-Lipschitz continuous on $\R_+$ and satisfies $\varphi'_{0_k}(s)+|\varphi_k'(s)|=1$ for a.e. $s\in [0,T+V_k]$ (and  $\varphi'_{0_k}(s)+|\varphi_k'(s)|=0$ for $s>T+V_k$).  
 Therefore by Ascoli-Arzel\`a's Theorem, taking if necessary a subsequence which we still denote by  $(\varphi_{0_k},\varphi_k)_k$, it converges  uniformly on any compact interval  $[0,S]$ and pointwise on $\R_+$ to a  Lipschitz continuous function $(\varphi_0,\varphi)$ such that $\varphi'_{0}(s)+|\varphi'(s)|\le1$ for $s\ge0$. 

Let us show that $(\varphi_0,\varphi)$ is a  BV$_{loc}$  graph completion of $u$, possibly not feasible (namely, not verifying the equality $\varphi'_{0}(s)+|\varphi'(s)|=1$ a.e.). Clearly,  $\varphi_0$  is nondecreasing,   $\varphi_0(0)=0$ and $\lim_{s\to+\infty}\varphi_0(s)\le T$. In fact, let us prove that 
$$
\lim_{s\to+\infty}\varphi_0(s)=T. 
$$
For any $\varepsilon>0$ we show that there exists some $S_\varepsilon>0$ such that  $\varphi_0(s)>T-\varepsilon$ \, $\forall s\ge S_\varepsilon$. Let $T-\varepsilon<t_\varepsilon<T$ and define,  for every $k$,  $S_{\varepsilon,k}:=\sigma_k(t_\varepsilon)$. Notice that  
$$
S_{\varepsilon,k}= t_\varepsilon+Var_{[0,t_\varepsilon]}(u_k) \le  t_\varepsilon+V( t_\varepsilon)=: S_\varepsilon,
$$
so that $t_\varepsilon=\varphi_{0_k}(S_{\varepsilon,k})\le \varphi_{0_k}(S_\varepsilon)<T$. Therefore, for any $s\ge S_\varepsilon$, we obtain that
$$
\varphi_0(s)\ge \varphi_0(S_\varepsilon)=\lim_k\varphi_{0_k}(S_\varepsilon)\ge t_\varepsilon >T-\varepsilon
$$ 
and the limit above is proved.  For every $t\in[0,T[$,   by  \eqref{sk}
  there exist a subsequence $(\sigma_{k'}(t))_{k'}$ and $\sigma(t)\in [0,t+V(t)]$  such that $\lim_{k'}\sigma_{k'}(t)=\sigma(t)$. Therefore, by the uniform convergence of  $(\varphi_{0_k},\varphi_k)_k$ on  
$[0,t+V(t)]$, recalling   (\ref{Thst}), it follows that 
$$
(\varphi_0,\varphi)\circ\sigma(t)=\lim_{k'}(\varphi_{0_{k'}},\varphi_{k'})\circ\sigma_{k'}(t)=(t,u(t)).
$$
 Hence $(\varphi_0,\varphi)$ is a (possibly not feasible)  BV$_{loc}$  graph completion of $u$ on $[0,T[$. 
\vv\noindent
Let  $\xi_k:=\xi[\xb,\ub,\varphi_{0_k},\varphi_k,v\circ\varphi_{0_k}]$ and $\xi:=\xi[\xb,\ub,\varphi_0,\varphi ,v\circ\varphi_0]$  be the corresponding  solutions of (\ref{SEaff}). Clearly, $\xi_k=x_k\circ \varphi_{0_k}$.  We set
$$
 \tilde x (t):= \xi\circ\sigma(t) \quad \forall t\in[0,T[,
$$
so that $\tilde x$ is a BV$_{loc}$  graph completion solution (on $[0,T[$). Actually, $\tilde x(t)=x(t)$ for any $t\in[0,T[$, since
$$
x(t)=\lim_{k'}x_{k'}(t)=\lim_{k'} \xi_{k'}\circ \sigma_{k'}(t)=\xi\circ\sigma(t)=\tilde x(t),
$$
where we used the uniform convergence of $\xi_k$ to $\xi$ on $ [0,t+V(t)]$, guaranteed by Proposition \ref{xiv},  together with the pointwise convergence of $ \sigma_{k'}(t)$ to $\sigma(t)$.   
 
\vsm
 In order to conclude the proof that $x$ is a  BV$_{loc}$ graph completion solution, let us show that   $\lim_j(\xi, \varphi)(\tilde s_j)= (x,u)(T)$, where $(\tilde s_j)_j$ is the same as in (iii) of the Claim.  In view of Remark \ref{condST},  hypothesis (iii) implies that    
 $$
 |(\xi_k,\varphi_k)(\tilde s_j)- (\xi_k,\varphi_k)(S_k)|\le\tilde\varepsilon(j)
 $$
  with $S_k:=\sigma_k(T)=T+V_k$, for every $k>k_j$ ($\ge j$),  for some positive, decreasing sequence  $\tilde\varepsilon$ with $\lim_j\tilde\varepsilon(j)=0$. Notice that,  for every $j$, 
 $$
 \sup_{[0,\tilde s_j]}|(\xi,\varphi)(s)-(\xi_k,\varphi_k)(s)|\le\varepsilon_j(k)
 $$
  for  some  positive, decreasing sequence  $\varepsilon_j$ with $\lim_k\varepsilon_j(k)=0$, because of the uniform convergence of $(\xi_k,\varphi_k)$ to $(\xi,\varphi)$ on compact intervals. Hence we can define a sequence $(\hat k_j)_j\subset \N$ with $\hat k_j\ge k_j$ and such that $\varepsilon_j(k)\le1/j$ for all $k>\hat k_j$.
Taking into account that 
$(\xi_k,\varphi_k)(S_k)=(x_k,u_k)(T),$ for every $k>\hat k_j$, we get
\bel{3}\begin{array}{l}
|(\xi,\varphi)(\tilde s_j)- (x(T),u(T))|\le
|(\xi,\varphi)(\tilde s_j)-(\xi_k,\varphi_k)(\tilde s_j)|+\\ \, \\
|(\xi_k,\varphi_k)(\tilde s_j)-(x_k(T),u_k(T))|+|(x_k(T),u_k(T))-(x(T),u(T))| \le \\ \, \\
1/j+\tilde\varepsilon(j)+|(x_k(T),u_k(T))-(x(T),u(T))|,
\end{array}\eeq
where  
$\lim_j|(x_k(T),u_k(T))-(x(T),u(T))|=0$, being $k>k_j\ge j$.
Therefore $\lim_j (\xi,\varphi)(\tilde s_j)=(x(T),u(T))$.
\vsm
  At this point we can recover a feasible space-time control in $\U(T;\ub)$ by introducing the change of variable
 $$
 \eta(s):=\int_0^s\left[\varphi_0'(r)+|\varphi'(r)|\right]\,dr \quad \forall s\ge0, \quad \tilde V:=\lim_{s\to+\infty}\eta(s)-T\le +\infty,
 $$
 considering,  e.g. $s(\cdot): [0,T+\tilde V[\to\R_+$, the strictly increasing  right-inverse  of $\eta$  and defining
 $$
 (\tilde\varphi_0,\tilde\varphi, \tilde\psi,\tilde S):=(\varphi_0\circ s,\varphi\circ s,\psi\circ s, T+\tilde V).
 $$
Let us set $\tilde\xi:=\xi[\xb,\ub,\tilde\varphi_0,\tilde\varphi, \tilde\psi]$. Notice that $(\varphi_0,\varphi)$ is constant on any interval $[s_1,s_2]$ where  $\eta$ is constant, so that $(\tilde\varphi_0,\tilde\varphi,\tilde\xi)\circ\eta=(\varphi_0,\varphi,\xi)$. Hence  $(\tilde\varphi_0,\tilde\varphi)$ turns out to be a  feasible BV$_{loc}$  graph completion  of $u$ on $[0,T]$ with clock $\tilde\sigma:=\eta\circ\sigma$. Finally, $x$ is a BV$_{loc}$  graph completion solution such that $x=\tilde\xi\circ\tilde\sigma$. \qed

\section{Technical proofs}\label{proofs}
\subsection{Proof of Lemma \ref{Legc}}

(i)\, Since $u$  is a BV function,  the set   ${\mathcal T}\subset[a,b]$  of discontinuity points of  $u$ is countable and right and left limits of $u$ always exist.  For every $\tau_j\in{\mathcal T},$ owing to the Whitney property, we can define the maps
$\tilde u_j^-$, $\tilde u_j^+$, $\tilde u_b:[0,1]\to U$ verifying 
$$
\tilde u_j^-(0)=u(\tau_j^-),  \ \tilde u_j^-(1)=u(\tau_j);\quad  \tilde u_j^+(0)=u(\tau_j),  \ \tilde u_j^+(1)=u(\tau_j^+);
$$
$$
\tilde u_b(0)=u(b),  \ \tilde u_b(1)=\bar u_1
$$
and such that 
$$
Var_{[0,1]}(\tilde u_j^-)\le C|  u(\tau_j)-u(\tau_j^-)|; \quad Var_{[0,1]}(\tilde u_j^+)\le C|  u(\tau_j^+)-u(\tau_j)|;
$$ 
$$
Var_{[0,1]}(\tilde u_b)\le C|  u(b)-\bar u_1|.
$$
We introduce the function  $\sigma:[a,b]\to [0,\lambda]$ given by 
$$
\sigma(t)=t-a+Var_{[a,t]}(u) \quad {\rm and}\quad \lambda:=b-a+V.
$$
Notice that  $u$ is continuous, [left-continuous, right-continuous] at   $t$ if and only if $\sigma$ is continuous, [left-continuous, right-continuous] at $t$ and let  $\hat\varphi_0$ be the unique, increasing  and continuous function such that $\hat \varphi_0\circ \sigma(t)=t$
for all $t\in[a,b]$.  Similarly to the proof of  \cite[Theorem 2.4]{ AR},  
let us set 
\bel{hat}
\hat\varphi(\sigma)=\left\{\begin{array}{l}
\tilde u_j^{-}\left(\frac{\sigma-\sigma(\tau_j^-)}{\sigma(\tau_j)-\sigma(\tau_j^-)}\right)\quad {\rm if} \,\,\sigma(\tau_j^-)<\sigma(\tau_j) \,{\rm and}\,\, \sigma\in[\sigma(\tau_j^-),\sigma(\tau_j)]
\\\,\\
\tilde u_j^+\left(\frac{\sigma-\sigma(\tau_j)}{\sigma(\tau_j^+)-\sigma(\tau_j)}\right)\quad {\rm if} \,\,\sigma(\tau_j)<\sigma(\tau_j^+) \,{\rm and}\,\, \sigma\in[\sigma(\tau_j),\sigma(\tau_j^+)]
\\ \,\\
u(\tau_j)\quad {\rm if, \, for\, some }\, j,\,{\rm either }\,\,\sigma=\sigma(\tau_j^-)=\sigma(\tau_j)\,\,{\rm or}\,\,\sigma=\sigma(\tau_j)=\sigma(\tau_j^+)
\\\,\\
\tilde u_b(\sigma-\lambda)\chi_{_{[\lambda,\lambda+1]}}(\sigma)+\tilde u_b(2-\sigma+\lambda)\chi_{_{(\lambda+1,\lambda+2]}}(\sigma)\quad {\rm if}\,\, \sigma\in[\lambda,\lambda+2]\\ \,\\
u\circ\hat\varphi_0(\sigma)\quad {\rm if}\,
\sigma\in[0,\lambda]\setminus\sigma({\mathcal{T}}).
\end{array}\right.
\eeq
Setting $\hat\varphi_0(\sigma)=b$ for $\sigma\in[\lambda,\lambda+2]$ we have that the function $(\hat\varphi_0,\hat\varphi):[0,\lambda+2]\to[a,b]\times U$ is absolutely continuous, verifies $(\hat\varphi_0,\hat\varphi)(0)=(a,u(a))$, $(\hat\varphi_0,\hat\varphi)(\lambda)=(\hat\varphi_0,\hat\varphi)(\lambda+2)=(b,u(b))$ and $(\hat\varphi_0,\hat\varphi)(\lambda+1)=(b,\bar u_1)$. Moreover,   
$$
\lambda\le Var_{[0,\lambda]}(\hat\varphi_0,\hat\varphi), 
$$
and
$$
Var_{[0,\lambda+2]}(\hat\varphi_0,\hat\varphi)\le(b-a)+2C(V+|u(b)-\bar u_1|)).
$$
 Let us now introduce, for $\sigma\in [0,\lambda+2],$ the arc-length parametrization
\bel{arc}
s(\sigma)=\int_0^\sigma(\hat\varphi'_0(r)+|\hat\varphi'(r)|)\,dr
\eeq
and let us set 
\bel{ss}
S:=s(\lambda+1)\quad{\rm and}\,\,\,\,\tilde S:=s(\lambda+2),
\eeq
so that
 $$
(b-a)+V+|u(b)-\bar u_1|\le  S\le\tilde S \le(b-a)+2C(V+|u(b)-\bar u_1|).
$$
Let $\tilde\sigma:[0,\tilde S]\to[0,\lambda+2]$ denote the inverse function of $s(\cdot)$ and define
\bel{fin}
(\varphi_0,\varphi)(s):=(\hat\varphi_0,\hat\varphi)\circ \tilde\sigma(s)\quad
{\rm for}\,\,s\in[0,\tilde S].
\eeq
We get   $\varphi'_0+|\varphi'|=1$ a.e.,   $(\varphi_0,\varphi)(0)=(a,u(a)),$  
$$
(\varphi_0,\varphi)(s(\lambda))=(\varphi_0,\varphi)(\tilde S)= (b,u(b)), \ 
(\varphi_0,\varphi)(S)=(b,\bar u_1), 
$$
and it is easy to see that for any $t\in [a,b]$ there is $s\in[0,\tilde S]$   (in fact, $s\in[0,s(\lambda)]$) such that
\medskip
 $(t,u(t))=(\varphi_0,\varphi)(s)$. 
\newline
\noindent (ii)\, For $s>\tilde S$, let us consider the periodic extension of the restriction $(\varphi_0,\varphi)$ to the interval $[s(\lambda), s(\lambda+2)]$, with period $p=s(\lambda+2)-s(\lambda)$. 
Setting, for every $j\ge1$,  $s_j:=s(\lambda+1)+j p,$ one clearly has  $(\varphi_0,\varphi)(s_j)=(b, \bar u_1)$ for all $j$, so proving (ii). \qed

 \subsection{Proof of Proposition \ref{xiv}}
 
  Let $(\varphi_0,\varphi,\psi)$, $(\varphi_{0_h},\varphi_h,\psi)\in {\mathcal U}(T; \ub,S)$, $\xi$, $\xi_h$ be  the given space-time controls and the corresponding solutions, respectively.  Since $\varphi'_0(s)+|\varphi'(s)|=1$ and $\varphi'_{0_h}(s)+|\varphi_h'(s)|=1$ a.e. on $[0,S]$,  so that in particular they are bounded, by standard estimates it follows that  
$$
\sup_{s\in[0,\tilde S]}|\xi(s)|,\ \sup_{s\in[0,\tilde S]}|\xi_h(s)|\le \bar M:=(|\xb|+(m+1)M\tilde S)e^{(m+1)M\tilde S}.
$$
Let us denote by $\omega$ a modulus of continuity of $g_0$ and by $\tilde M$,  $\tilde L$   a sup-norm and a Lipschitz constant, respectively,   for the vector fields $g_i$, $i=0,\dots m$,   in the compact set $\overline{B_n(0,\bar M)}\times U\times V$.  

Let us start by showing that
$\xi=\xi[\xb,\ub,\varphi_0,\varphi,\psi]\equiv\xi[\xb,\ub,\varphi_0,\varphi,v\circ\varphi_0]=:\tilde\xi$.
Indeed,  there is an at most countable number of disjoint intervals, say $[s^1_j,s^2_j]$ for $j\in J$,  where $\varphi_0$ is constant; moreover,  $\psi$ may differ from $v\circ\varphi_0$ only on these intervals, for $\varphi_0 ^{\leftarrow}(\varphi_0(s))$ is single valued outside such set.  Hence,  for every $s\in]0,\tilde S]$, we get
 $$
 \begin{array}{l}
\xi(s)-\tilde \xi(s)=
 \int_{[0,s]\setminus \bigcup_j [s^1_j,s^2_j]}[g_0(\xi(r),\varphi(r),\psi(r))- g_0(\tilde\xi(r),\varphi(r),v\circ\varphi_0(r))]\,dr  \\
+\int_{[0,s]}\sum_{i=1}^m[g_i(\xi(r),\varphi(r))-g_i(\tilde\xi(r),\varphi(r))]\varphi'_i(r)]\,dr
\end{array}
  $$
and thesis \eqref{xiv1} follows easily by Gronwall's Lemma.

In order to prove \eqref{xivh},  for every $s\in[0,\tilde S]$ we apply again   Gronwall's Lemma and get   
\bel{4r}
\begin{array}{l}
 |\xi_h(s)-\xi(s)|\le \left(\left| \int_0^{s}[g_0(\xi(r),\varphi(r),v\circ\varphi_{0}(r))[\varphi'_{0_h}(r)-\varphi'_0(r)]+\sum_{i=1}^mg_i(\xi(r),\varphi(r))[\varphi'_{h}(r)-\varphi'(r)]]\,dr\right| + \right. \\ \, \\
  \left.  
 (m +1)\tilde L\int_0^{\tilde S}|\varphi_h(r)-\varphi(r)|\,dr+ \int_0^{\tilde S} \omega (|v\circ\varphi_{0_{h}}(r)-v\circ\varphi_{0}(r)|)\varphi'_{0_h}(r)\,dr\right)\,e^{(m+1)\tilde L{\tilde S}}.
 \end{array}
\eeq
The uniform convergence of $(\varphi_{0_h},\varphi_h)$ to $(\varphi_{0},\varphi)$ on $[0,\tilde S]$  implies that the maps $(\varphi'_{0_h},\varphi'_{h})$ tend to $(\varphi'_0, \varphi')$ in  the weak$^*$ topology   of $L^\infty([0,\tilde S], \R^{1+m})$, so that 
 $$
 f_h(s):=\left| \int_0^{s}[g_0(\xi(r),\varphi(r),v\circ\varphi_{0}(r))[\varphi'_{0_h}(r)-\varphi'_0(r)]+\sum_{i=1}^mg_i(\xi(r),\varphi(r))[\varphi'_{h}(r)-\varphi'(r)]]\,dr\right|
 $$ 
 tends to 0 as $h\to+\infty$.    The uniform convergence to 0 of the $f_h$'s now follows from Ascoli-Arzel\'a Theorem, for the $f_h$'s are equibounded and equi-Lipschitzean. 
The  convergence to 0 of the second integral  in the r.h.s. of (\ref{4r}) is trivial. It remains to prove the convergence to $0$, eventually for a further subsequence,   of the last term of (\ref{4r}). Let us set  $\sigma_h:=\varphi_{0_h}^{-1}$ and observe that  
  \bel{nonlip}
\int_0^{\tilde T}|v(t)-v\circ\varphi_{0}\circ\sigma_h(t)|\,dt=\int_0^{\tilde S}|v\circ\varphi_{0_{h}}(s)-v\circ\varphi_{0}(s)|\varphi'_{0_h}(s)\,ds.
\eeq
Now, it suffices to prove that   the expression in  (\ref{nonlip}) tends to 0 as $h\to+\infty$: in this case, indeed,   there exists a subsequence of $(v-v\circ\varphi_{0}\circ\sigma_h)$ converging to 0 a.e.  on  $[0,\tilde T]$, and the Dominated Convergence Theorem implies that, for such subsequence, 
\bel{nonlipg}
\int_0^{\tilde S}\omega (|v\circ\varphi_{0_{h}}(s)-v\circ\varphi_{0}(s)|)\varphi'_{0_h}(s)\,ds=\int_0^{\tilde T}\omega(|v(t)-v\circ\varphi_{0}\circ\sigma_h(t)|)\,dt\to0, 
\eeq
so implying  \eqref{xivh}. 

\noindent Since $|\varphi'_{0_h}|\le 1$, when $v$ is a continuous function (\ref{nonlip})  holds true,  
owing to the uniform continuity of $v$ and to the uniform convergence of $\varphi_{0_h}$ to $\varphi_{0}$ on $[0,\tilde T]$.  For  $v\in L^1([0,\tilde T], V)$, $\forall\varepsilon>0$   there exists, by density, $\tilde v\in C_c([0, \tilde T], \R^l)$  such that $\int_0^{\tilde T}|\tilde v(t)-v(t)|\,dt\le\varepsilon.$ Hence we get
$$
\begin{array}{l}\int_0^{\tilde S}| v\circ\varphi_{0_h}(s)-v\circ\varphi_{0}(s)|\varphi'_{0_h}(s)ds\le 
\int_0^{\tilde S}| v\circ\varphi_{0_h}(s)-\tilde v\circ\varphi_{0_h}(s)|\varphi'_{0_h}(s)\,ds+\\ \, \\
\quad\int_0^{\tilde S}|\tilde v\circ\varphi_{0_h}(s)-\tilde v\circ\varphi_{0}(s)|\varphi'_{0_h}(s)\,ds+
\int_0^{\tilde S}|\tilde v\circ\varphi_{0 }(s)- v\circ\varphi_{0}(s)|\varphi'_{0_h}(s)\,ds.
\end{array}
$$
Performing the change of variable $t=\varphi_{0_h}(s)$,   the first integral on the r.h.s. is smaller than $\varepsilon$, while 
the second one converges to 0 because $\tilde v$ is continuous. For the third integral on the r.h.s., taking into account that  $|v(t)|$, $|\tilde v(t)|\le \hat M$ for all $t\in[0,\tilde T]$ for some  $\hat M>0$,  by the weak$^*$  convergence of  $\varphi'_{0_h}$ to   $\varphi'_0$  we derive that
  $$
 \begin{array}{l}
 \int_0^{\tilde S}|\tilde v\circ\varphi_{0 }(s)- v\circ\varphi_{0}(s)|\varphi'_{0_h}(s)\,ds\to \int_0^{\tilde S}|\tilde v\circ\varphi_{0}(s)-v\circ\varphi_{0}(s)|\varphi'_{0}(s)\,ds \quad \text{as $h\to+\infty$,}\end{array}
$$
and the last term is smaller  is smaller than $\varepsilon$   by the change of variable  $t=\varphi_{0}(s)$.  This concludes the proof of (\ref{nonlip}) by the arbitrariness of $\varepsilon>0$. \qed
\vv 
\subsection{Proof of Theorem \ref{Cl1-2}}
 
\noindent{\sc \underline{Case 1}:} $\lim_{t\to T^-}\sigma(t)=\bar S<+\infty$.  
Let us extend  $\sigma$ to   $[-T,2T]$ as follows:
\bel{ext2}\tilde\sigma(t) =\sigma(t) \chi_{_{|[0,T]}}(t) -\sigma(-t)\chi_{_{|[-T,0[}}(t)-(\sigma(2T-t)-2\sigma(T))
\chi_{_{|]T,2T]}}(t).\eeq
Let   $\rho:\R\to\R_+,$  $\rho\in{\mathcal C}^\infty$ be an even map, with compact support contained on $[-T,T]$ and  such that $\int_\R\rho(t)\,dt=1$; for $h\in\N$ let us set  $\rho_h(t):=2h\rho(2ht)$  and
\bel{conv1}
\tilde\sigma_h(t):= \int_{-\infty}^{+\infty}\tilde\sigma(t-\tau)\rho_h(\tau)\,d\tau.
\eeq
 The fact that  $\rho$ is even together with (\ref{ext2}) easily yield,
for every $h\in\N$,  
 \bel{mean0}
 \tilde\sigma_h(0)=0,\quad\   \tilde\sigma_h(T)=\sigma(T^-)=\bar S.
 \eeq
 By construction, the $\tilde\sigma_h$ are continuous, strictly increasing,  and, by a property of the convolution product, 
 \bel{mean}
 \lim_h\tilde\sigma_h(t)=\frac{\sigma(t^+)+\sigma(t^-)}{2}\quad {\rm for }\ 0\le t< T. 
 \eeq 
It is easy to  show  that for any $t_1,\,t_2\in [0,T[$ with $t_1<t_2$, (\ref{ext2}) implies 
\bel{1-2}
\begin{array}{l}\tilde\sigma_h(t_2)-\tilde\sigma_h(t_1)=\int_{-\infty}^{+\infty}(\tilde\sigma(t_2-\tau)-\tilde\sigma(t_1-\tau))\rho_h(\tau)\,d\tau\ge\\
\qquad\qquad\qquad\qquad\qquad\qquad\qquad\int_{-\infty}^{+\infty}(t_2-t_1)\rho_h(\tau)\,d\tau\ge t_2-t_1.
\end{array}\eeq

 Let    $(\bar t_h)_h$ be a strictly increasing sequence of continuity points of $\sigma$ converging to $T.$ By
the strict monotonicity of $\sigma$ and \eqref{mean} it follows that $s_h:=\tilde\sigma_h(\bar t_h) < \bar S$ and $\lim_h s_h=\bar S$.   
In order to obtain a sequence of strictly increasing maps which are onto on $\R_+$ and converging to $\sigma,$  let us set
  $$
 \sigma_h(t):=\left\{\begin{array}{l}\tilde\sigma_h(t)\quad {\rm for}\, \,t\le \bar t_h\\
s_h\sqrt{\frac{T-\bar t_h}{T-t}}\quad\,{\rm for}\,\,\bar t_h\le t<T.
\end{array}\right.
$$
Since $\sigma'_h(t)=s_h\frac{\sqrt{T-\bar t_h}}{2(T-t)^{3/2}}$ for $t\in]\bar t_h,T[$,   $\sigma'_h \ge  \frac{s_h}{2(T-\bar t_h)} \ge 1$ for $h$ large enough and for any $t_1,\,t_2\in [0,T[$ with $t_1<t_2,$ we get
$\sigma_h(t_2)-\sigma_h(t_1)\ge t_2-t_1$. Moreover,   the maps $ \sigma_h$ are continuous,  onto on $\R_+$,  and  verify (\ref{mean}) for every $t<T$, since $  \sigma_h(t)=\tilde\sigma_h(t)$ for all $h$ such that $\bar t_h>t$.
 The inverse functions  
$$
\varphi_{0_h}(s):=\sigma_h^{-1}(s)=\left\{\begin{array}{l}\tilde\sigma_h^{-1}(s),\quad 0\le s\le s_h\\
T-\frac{s_h^2}{ s^2}(T-\bar t_h),\quad s> s_h,
\end{array}\right.
$$
are  $1$-Lipschitz continuous and strictly increasing, so that by Ascoli-Arzel\`a's Theorem, taking if necessary a subsequence, they converge     uniformly on any compact interval  $[0,S]$ and pointwise on $\R_+$ to a  $1-$Lipschitz continuous function $\hat\varphi_0$. In fact, $\hat\varphi_0=\varphi_0$, where $\varphi_0=\sigma^{-1}$ on $[0,\bar S[$ and $\varphi_0(s)=T$ for all $s\ge\bar S$. Indeed, if $t<T$ is a continuity point of $\sigma$,   $\sigma_h(t)= \tilde\sigma_h(t)<\bar S$ for $h$ sufficiently large, and 
$$
t=\varphi_{0_h}(\tilde\sigma_h(t))\le |\varphi_{0_h}(\tilde\sigma_h(t))-\varphi_{0_h}(\sigma(t))|+\varphi_{0_h}(\sigma(t))\le
|\tilde\sigma_h(t)-\sigma(t)|+\varphi_{0_h}(\sigma(t))
$$
which implies that   
$$
 \varphi_0(\sigma(t))=t=\lim_h \varphi_{0_h}(\tilde\sigma_h(t)) =\hat\varphi_0(\sigma(t)).
$$
 If $t$ is not a continuity point, then
there exist two sequences $t_k^1$ and $t_k^2$ of continuity points of  $\sigma$ with 
$$t_k^1<t<t_k^2,\quad t_k^1\to t,\quad  t_k^2\to t.$$ 
Since the $\varphi_{0_h}$ are increasing,  then $\hat\varphi_0$ is increasing and
\bel{in}
\hat\varphi_0(\sigma(t_k^1))\le\hat\varphi_0(\sigma(t))\le\hat\varphi_0(\sigma(t_k^2)).\eeq 
Since $t_k^1$, $t_k^2$ are continuity points we have $t_k^i=\hat\varphi_0(\sigma(t_k^i))=\varphi_0(\sigma(t_k^i))$ for $i=1,2$ and  (\ref{in}) implies 
 $$ 
 t_k^1\le\hat\varphi_0(\sigma(t))\le t_k^2.
 $$
Passing to the limit,   we can conclude that $\hat\varphi_0=\varphi_0$ on $[0,\bar S[$. 

\noindent  Moreover,  $\varphi_{0_h}(s)\ge \varphi_{0_h}(\bar S)\ge \varphi_{0_h}(s_h)$ for every $s\ge \bar S$ and, setting $t_h:=\varphi_{0_h}(\bar S)$, we get  $\varphi_{0_h}(s)\ge t_h\ge \bar t_h$ for every $s\ge \bar S$ and  
$$
\sup_{s\ge \bar S} |\varphi_{0_h}(s)-\varphi_0(s)|=\sup_{s\ge \bar S}[T-\varphi_{0_h}(s)] \le(T-t_h)\le (T- \bar t_h)\to 0 \ \ \text{as $h\to+\infty$.}
$$
Hence $\varphi_{0_h}$ converges uniformly to $\varphi_0$  on $\R_+$ and  we have
\bel{sup1}
\sup_{s\in\R_+} |\varphi_{0_h}(s)-\varphi_0(s)|\le\varepsilon(h)+(T-  t_h)
\eeq
where $\varepsilon(h):=\sup_{s\in[0,\bar S]}|\varphi_{0_h}(s)-\varphi_0(s)|$.
\vsm
By (\ref{mean}) the proof is concluded if $\sigma(t)=\frac{\sigma(t^+)+\sigma(t^-)}{2}$ for every $t\in[0,T[$.  In the general case,  we can adapt the above construction simply  by replacing the sequence  $(\tilde\sigma_h)_h$   on $[0,T[$ by a new sequence of strictly increasing functions, pointwisely converging to the extended map $\sigma:[0,T]\to[0,\bar S]$, $\sigma(T)=\bar S$,  and verifying  (\ref{mean0}) and (\ref{1-2}), whose existence easily follows  by \cite[Theorem 5.1]{AR}.

\vsm 
\noindent{\sc \underline{Case 2}:}  $\lim_{t\to T^-}\sigma(t)=+\infty$.
 \,\,The function $\sigma$ does not in general belong to $L^1(T)$, hence the  convolution product \eqref{conv1} cannot be defined as in the previous   case.  Let us choose a strictly increasing sequence $(\bar t_i)_i$ (with $\bar t_0:=0$) of continuity points of $\sigma,$  such that $\lim_i\bar t_i=T.$
We know  that $\sigma$ is monotone and $\sigma\in L^1_{loc}(T)$ and we can perform the convolution of the restriction
$\sigma^i:=\sigma_{|_{I_i}}$, where  $I_i:=[\bar t_{i-1},\bar t_i]$  and $|I_i|:=\bar t_i- \bar t_{i-1}$ for $i\ge1$.
\newline
Let     $\rho^i :\R\to\R_+$ be an even,  ${\mathcal C}^\infty$  function, with compact support contained on $[-|I_i|, |I_i|], $  such that $\int_\R\rho^i (t)\,dt=1$ and let us set  $\rho^i_h(t):=2h\rho^i(2ht)$. Let us     extend  
each function $\sigma^i$ to $[\bar t_{i-1}-|I_i| ,\bar t_i+|I_i| ]$ as follows:   for $0<t\le |I_i|$   and for every $i\ge 1$ 
we set
\bel{ext1}\begin{array}{l}
\tilde\sigma^i(\bar t_{i-1}-t):=-\sigma^i(\bar t_{i-1}+t)+2\sigma(\bar t_{i-1}) \\
\tilde\sigma^i(\bar t_{i}+t):=-\sigma^i(\bar t_{i}-t)+2\sigma(\bar t_{i}).\end{array}\eeq
Let us now define for each $i$ and $h\ge 1$
$$\tilde\sigma_h^i(t):=\int_{-\infty}^{+\infty}\tilde\sigma^i(t-\tau)\rho^i_h(\tau)\,d\tau.$$
The fact that  $\rho^i$ is even and (\ref{ext1}) easily yield,
for every $h,\,i\in\N$,  
\bel{mean01}
\tilde\sigma^1_h(0)=0,\quad \tilde\sigma^{i}_h(\bar t_{i-1})=\sigma(\bar t_{i-1}),\quad  \tilde\sigma^i_h(\bar t_i)=\sigma(\bar t_i).
\eeq
We set for $t\in[0,T[$
\bel{convi}
\tilde\sigma_h(t):=\tilde\sigma_h^i(t),\quad {\rm if}\,t\in[\bar t_{i-1},\bar t_i[
\eeq
 so that $\tilde\sigma_h(\bar t_i)=\sigma(\bar t_i)$ for every $h$ and $i.$
By construction, $\tilde\sigma_h$ is continuous on $[0,T[$, strictly increasing since $\sigma$ is so, and for $t\in[0,T[$
$$\lim_{h\to+\infty}\tilde\sigma_h(t)=\frac{\sigma(t^+)+\sigma(t^-)}{2}$$
 Moreover if $0\le t_1<t_2<T$ then it is not difficult to prove, that
\bel{mon}\tilde\sigma_h(t_2)-\tilde\sigma_h(t_1)\ge t_2-t_1.\eeq
Indeed if $t_1,t_2\in I_i$ for some $i$,      we can prove that
\bel{1-2bis}
\begin{array}{l}\tilde\sigma_h(t_2)-\tilde\sigma_h(t_1)=
\tilde\sigma_h^i(t_2)-\tilde\sigma_h^i(t_1)=\int_{-\infty}^{+\infty}(\tilde\sigma^i(t_2-\tau)-\tilde\sigma^i(t_1-\tau))\rho_h^i(\tau)\,d\tau\ge\\\\
\int_{-\infty}^{+\infty}(t_2-t_1)\rho_h^i(\tau)\,d\tau\ge t_2-t_1.\end{array}
\eeq
If $t_1\in I_j$ and $t_2\in I_i$ and $j\ne i$, the same result can be easily proved, by interpolating 
a suitable number of $\sigma(\bar t_k)=\tilde\sigma_h(\bar t_k)$, since each $\tilde\sigma_h$ is continuous 
and  obtained by piecing together the ${\tilde\sigma_h^i}$ restricted to $I_i$.
\newline Since $\tilde\sigma_h$ is increasing, defined on $[0,T[$ onto $\R_+$ and (\ref{1-2bis}) holds, the maps $\tilde\varphi_{0_h}:=\tilde\sigma_h^{-1}:\R_+\to[0,T[$ are strictly increasing, surjective and$1$-Lipschitz continuous,   so that $\lim_{s\to\infty}\tilde\varphi_{0_h}(s)=T.$   
Taking if necessary a subsequence,  $(\tilde\varphi_{0_h})_h$ converges locally  uniformly  to an increasing  $1-$Lipschitz continuous function $\tilde\varphi_0$, which can be proven to coincide  with $\varphi_0$,  arguing similarly to  the previous case.   Hence for each   $t\in[0,T[$, ($\sigma(t^+)<+\infty$ and)  we can write
\bel{sup2}
\begin{array}{l}
\sup_{s\in\R_+}|\tilde\varphi_{0_h}(s)-\varphi_0(s)|\le \sup_{s\in[0,\sigma(t^+)]}|\tilde\varphi_{0_h}(s)-\varphi_0(s)|\\\,\\
\qquad \quad  +\sup_
{s\ge \sigma(t^+)}|\tilde\varphi_{0_h}(s)-\varphi_0(s)|\le
 \varepsilon_t(h)+(T-(t_h\wedge t)),
\end{array}
\eeq
where, setting $\varepsilon_t(h):=\sup_{s\in[0,\sigma(t^+)]}|\tilde\varphi_{0_h}(s)-\varphi_0(s)|$ and  $t_h:=\tilde\varphi_{0_h}(\sigma(t^+))$, one has
$$
|t_h-t|\le \varepsilon_t(h)  \ \ \text{and} \ \   \lim_h\varepsilon_t(h)=0.
$$
Finally,   we  recover a new sequence,  denoted by $(\sigma_h)_h$ with  strictly increasing, $1$-Lipschitz continuous inverse functions  $\varphi_{0_h}$ verifying (\ref{sup2}) and   such that $\lim_h\sigma_h(t)=\sigma(t)$   at every  $t\in[0,T[$. Since  $\sigma([0,T[)=\R_+$, differently from the previous case, we cannot apply  straightforwardly \cite[Theorem 5.1]{AR}, but we can  adapt the arguments of its proof  to unbounded maps. 
Let ${\mathcal T}\subset[0,T[$ be the (countable)  set of discontinuity points of  $\sigma$. For every $\tau_j\in{\mathcal T}$, set  
$s_{1,j}:=\lim_{\tau\to \tau_j^-}\sigma(\tau)$ and $s_{2,j}:=\lim_{\tau\to \tau_j^+}\sigma(\tau)$ and define a new sequence $(\varphi_{0_h})_h$
such that $\varphi_{0_h}(s)=\tilde\varphi_{0_h}(s)$ for every $s\notin \cup_{j}[s_{1,j},s_{2,j}]$, while $\varphi_{0_h}(s)$   is a suitable strictly  increasing, $1$-Lipschitz  function obtained,  in each interval $[s_{1,j},s_{2,j}],$ by two  concatenated linear interpolations of values of $\tilde \varphi_{0_h}$, with range equal to the interval $\left[\tilde\varphi_{0_h}(s_{1,j}),
\tilde\varphi_{0_h}(s_{2,j})\right]$ and   such that the inverse functions $\sigma_h:=\varphi_{0_h}^{-1}$ verify $\lim_h\sigma_h(\tau_j)=\sigma(\tau_j)$ for every $j$ (we refer for the precise construction to the proof of  \cite[Theorem 5.1]{AR}).  At this point, it is not difficult to see that $(\varphi_{0_h})_h$, as $(\tilde\varphi_{0_h})_h$,  converges locally uniformly to $\varphi_0$ and  verifies (\ref{sup2}).
\newline In order to show that  $\sigma_h$ converges pointwisely to $\sigma$, let us consider 
the   sequence  $(\bar t_i)_i$ of continuity points of $\sigma$, converging to $T$, which was used in the definition (\ref{convi}),     and set $s_i:=\sigma(\bar t_i).$
By construction,  for all $h$ and   $i$,  we have 
$$
\varphi_{0_h}(s_i)=\tilde\varphi_{0,h}(s_i)=\varphi_{0}(s_i)=\bar t_i,
$$
so that $\sigma_h(\bar t_i)=\tilde\sigma_h(\bar t_i)=\sigma(\bar t_i)$ and  $\sigma_h([0,\bar t_i])=[0,s_i]$. Hence   the  sequence $(\sigma_h)_h$ restricted to $[0,\bar t_i]$ verifies $\lim_h\sigma_h(t)=\sigma(t)$ for  $t\in[0,\bar t_i]$ by the proof of \cite[Theorem 5.1]{AR}.
Since,  for every $t\in[0,T[$  there is some $i$ such that $t\in [0,\bar t_i]$,   we can conclude that   $\sigma_h$ pointwisely converges  to $\sigma$  on the whole interval $[0,T[$.  \qed

\end{document}